\documentclass[11pt]{amsart}
\usepackage{amsmath,amsthm,hyperref,enumitem,mathtools}
\usepackage{color,todonotes,comment}
\usepackage[alphabetic,initials,nobysame]{amsrefs}

\BibSpec{collection.article}{%
	+{}  {\PrintAuthors}			    {author}
	+{,} { \textit}                     {title}
	+{.} { }                            {part}
	+{:} { \textit}                     {subtitle}
	+{,} { \PrintContributions}         {contribution}
	+{,} { \PrintConference}            {conference}
	+{}  {\PrintBook}                   {book}
	+{,} { }                            {booktitle}
	+{,} { }							{series}
	+{,} { }							{publisher}
	+{,} { }							{address}
	+{,} { \PrintDateB}                 {date}
	+{,} { pp.~}                        {pages}
	+{,} { }                            {status}
	+{,} { \PrintDOI}                   {doi}
	+{,} { available at \eprint}        {eprint}
	+{}  { \parenthesize}               {language}
	+{}  { \PrintTranslation}           {translation}
	+{;} { \PrintReprint}               {reprint}
	+{.} { }                            {note}
	+{.} {}                         	{transition}
	+{}  {\SentenceSpace \PrintReviews} {review}
}

\title[Lipschitz-Volume rigidity and Sobolev coarea inequality]{Lipschitz-Volume rigidity and Sobolev coarea inequality for metric surfaces}

\author{Damaris Meier}
\address[Damaris Meier]{Department of Mathematics\\ University of Fribourg\\ Chemin du Mus\'ee 23\\ 1700 Fribourg, Switzerland.}
\email{damaris.meier@unifr.ch}

\author{Dimitrios Ntalampekos}
\address[Dimitrios Ntalampekos]{Mathematics Department, Stony Brook University, Stony Brook, NY 11794, USA.}
\email[Corresponding author]{dimitrios.ntalampekos@stonybrook.edu}

\thanks{The first-named author is partially supported by UniFr Doc.Mobility Grant DM-22-10. The second-named author is partially supported by NSF Grant DMS-2000096.}

\date{\today}

\newtheorem{thm}{Theorem}[section]
\newtheorem{lemma}[thm]{Lemma}
\newtheorem{corollary}[thm]{Corollary}
\newtheorem{proposition}[thm]{Proposition}

\theoremstyle{remark}
\newtheorem{remark}[thm]{Remark}
\newtheorem{question}[thm]{Question}

\theoremstyle{definition}
\newtheorem{example}[thm]{Example}

\newcommand{\R}{\mathbb{R}}
\newcommand{\N}{\mathbb{N}}
\newcommand{\D}{\mathbb{D}}

\newcommand{\hm}{{\mathcal H^2}}
\newcommand{\loc}{\mathrm{loc}}
\newcommand{\lip}{\mathrm{Lip}}
\renewcommand{\mod}{\operatorname{Mod}}
\renewcommand{\bar}[1]{\overline{#1}}
\renewcommand{\top}{\mathrm{top}}

\DeclareMathOperator{\vol}{Vol}
\DeclareMathOperator{\dist}{dist}
\DeclareMathOperator{\inter}{int}
\DeclareMathOperator{\diam}{diam}
\DeclareMathOperator{\ap}{ap}
\DeclareMathOperator*{\essinf}{essinf} 
\DeclareMathOperator{\md}{md}

\numberwithin{equation}{section}

\begin{document}
\maketitle
\begin{abstract}
    We prove that every $1$-Lipschitz map from a closed metric surface onto a closed Riemannian surface that has the same area is an isometry. If we replace the target space with a non-smooth surface, then the statement is not true and we study the regularity properties of such a map under different geometric assumptions. Our proof relies on a coarea inequality for continuous Sobolev functions on metric surfaces that we establish, and which generalizes a recent result of Esmayli--Ikonen--Rajala. 
\end{abstract}

\section{Introduction}

The Lipschitz-Volume rigidity problem in its general formulation asks whether every $1$-Lipschitz and surjective map between metric spaces that have the same volume (e.g.\ arising from Hausdorff measure) is necessarily an isometry. It is well-known that the answer to this problem is affirmative for maps between smooth manifolds.

\smallskip
\noindent
\textit{Let $X,Y$ be closed Riemannian $n$-manifolds, where $n\geq 1$. If $\vol(X)=\vol(Y)$, then every $1$-Lipschitz map from $X$ onto $Y$ is an isometric homeomorphism.}
\smallskip

See \cite{BuragoIvanov}*{Section 9} or \cite{BessonEtAl}*{Appendix C} for a proof of this fact.  Moreover this statement has been generalized to singular settings of Alexandrov and limit RCD spaces by Storm \cite{Storm}, Li \cite{Li:alexandrov}, and Li--Wang \cite{LiWang}. See also \cite{Li:survey} for an overview of the Lipschitz-Volume rigidity problem. The problem in the setting of integral current spaces has been recently studied by Basso--Creutz--Soultanis \cite{BCS}, Del Nin--Perales \cite{DNP}, and Züst \cite{Zus23}.

The recent developments in the uniformization of non-smooth metric surfaces by Rajala, Romney, Wenger, and the current authors \cites{Raj:17,NR:21,MW21,NR22}, allow us to establish the above rigidity statement in the two-dimensional setting under no geometric, smoothness, or curvature assumptions on $X$. 

\begin{thm}\label{thm:main:riemannian}
    Let $X$ be a closed metric surface and $Y$ be a closed Riemannian surface. If $\mathcal H^2(X)=\mathcal H^2(Y)$, then every $1$-Lipschitz map from $X$ onto $Y$ is an isometric homeomorphism.  
\end{thm}

Here a closed metric surface is a compact topological $2$-manifold without boundary, equipped with a metric that induces its topology. Also, an isometric map is a distance-preserving map. We state an immediate corollary.

\begin{corollary}\label{corollary:sphere}
    Among all metrics $d$ on $\mathbb S^2$ that are at least as large as the spherical metric, the map $d\mapsto \mathcal H^2_d(\mathbb S^2)$ has a unique minimum attained by the spherical metric.
\end{corollary}

We note in the next example that the conclusion is not true in general if we replace the spherical metric with a non-smooth metric.

\begin{example}\label{example:sphere}
  Consider a non-constant rectifiable curve $E$ in $\mathbb S^2$ and let $d_0$ be the length metric $\chi_{\mathbb S^2\setminus E} \, ds + (1/2)\chi_{E}\, ds$. Then there exist infinitely many distinct metrics $d\geq d_0$ having the same area as $d_0$. Namely, for each $\delta\in (1/2,1]$, the metric $\chi_{\mathbb S^2\setminus E} \, ds + \delta\chi_{E}\, ds$ has this property.  
\end{example}

{
One of the most technical difficulties of Theorem \ref{thm:main:riemannian} is establishing the injectivity of the map in question; see Lemma \ref{lemma:injectivity}. 
Since this issue is not present in Corollary \ref{corollary:sphere}, it is conceivable that the result can be obtained in higher dimensions as well by a modification of our argument. 
}

\subsection{Area-preserving and Lipschitz maps between surfaces}
A map as in Theorem \ref{thm:main:riemannian} preserves the Hausdorff $2$-measure, or else area measure, of every measurable set. Theorem \ref{thm:main:riemannian} is a consequence of Theorem \ref{thm:main} below, which provides several topological and regularity results for area-preserving and Lipschitz maps between surfaces of locally finite Hausdorff $2$-measure. 

We provide the necessary definitions. Let $X$ and $Y$ be metric surfaces of locally finite Hausdorff $2$-measure. A  map $f\colon X\to Y$ is \emph{area-preserving} if $\hm(A)=\hm(f(A))$ for every measurable set $A\subset X$. A map $f\colon X\to Y$ is \textit{Lipschitz} if there exists $L>0$ such that for all $x_1,x_2\in X$ we have
$$d(f(x_1),f(x_2))\leq L\,d(x_1,x_2).$$
In this case, we say that $f$ is $L$-Lipschitz. A homeomorphism $f\colon X\to Y$ is \emph{quasiconformal} (abbr.\ QC) if there exists $K\geq 1$ such that $$K^{-1}\mod f(\Gamma)\leq \mod \Gamma \leq K \mod f (\Gamma)$$ for each path family $\Gamma$ in $X$; here $\mod$ refers to $2$-modulus. In this case we say that $f$ is $K$-quasiconformal. A map $f\colon X\to Y$ is a map of \textit{bounded length distortion} (abbr.\ BLD) if there exists a constant $K\geq 1$ such that $$K^{-1}\cdot\ell(\gamma)\leq\ell(f\circ\gamma)\leq K\cdot\ell(\gamma)$$  for all curves $\gamma$ in $X$; this includes curves of infinite length. In this case we say that $f$ is a map of $K$-bounded length distortion. 

We say that the surface $X$ is \textit{reciprocal} if there exists a constant $\kappa>0$ such that for every quadrilateral $Q\subset X$ and for the families $\Gamma(Q)$ and $\Gamma^*(Q)$ of curves joining opposite sides of $Q$ we have 
$$\mod  \Gamma(Q)\cdot \mod\Gamma^*(Q)\leq \kappa.$$
By a result of Rajala \cite{Raj:17}*{Section 14}, if a surface is reciprocal then the above holds for some $\kappa\leq (\pi/2)^2$. Reciprocal surfaces are important because they are precisely the metric surfaces that admit quasiconformal parametrizations by Riemannian surfaces \cites{Raj:17,Iko:19,NR22}. We say that $X$ is \textit{upper Ahlfors $2$-regular} if there exists $K>0$ such that
$${\mathcal H^2(B(x,r))}\leq K  r^2$$  
for every ball $B(x,r)\subset X$. If $X$ is (locally) upper Ahlfors $2$-regular, then it is also reciprocal \cite{Raj:17}. See Section \ref{section:reciprocal} for further details. We state our main theorem, which is also concisely presented in Table \ref{table}. 

\begin{thm} \label{thm:main}
Let $X,Y$ be metric surfaces without boundary and with locally finite Hausdorff $2$-measure, and let $f\colon X\to Y$ be an area-preserving surjective map. 
\begin{enumerate}[label=\normalfont(\arabic*)]
    \item If $X$ is reciprocal and $f$ is Lipschitz, then there exists a constant $K\geq 1$ such that 
   $$K^{-1}\cdot \ell(\gamma)\leq \ell(f\circ \gamma)\leq K\cdot \ell(\gamma)$$
   for all curves $\gamma$ in $X$ outside a curve family $\Gamma_0$ with $\mod\Gamma_0=0$. Moreover, if $f$ is $1$-Lipschitz, then $K=1$. \label{main:length} 
        
    \item If $Y$ is reciprocal and $f$ is Lipschitz, then there exists a constant $K\geq 1$ such that $f$ is a $K$-quasiconformal homeomorphism and  $${K}^{-1} \cdot\ell(\gamma)\leq\ell(f\circ\gamma)\leq K\cdot\ell(\gamma)$$
   for all curves $\gamma$ in $X$ outside a curve family $\Gamma_0$ with $\mod\Gamma_0=0$. Moreover, if $f$ is $1$-Lipschitz, then $K=1$.  \label{main:injective}
   
   \item If $Y$ is upper Ahlfors $2$-regular and $f$ is Lipschitz, then there exists a constant $K\geq 1$ such that $f$ is a homeomorphism of $K$-bounded length distortion. 
   \label{main:bld} 

\end{enumerate}
The constant $K$ in \ref{main:length}--\ref{main:bld} depends quantitatively on the assumptions.
\begin{enumerate}[label=\normalfont(\arabic*)]\setcounter{enumi}{3}
   \item If $Y$ is Riemannian and $f$ is $1$-Lipschitz, then $f$ is an isometric homeomorphism. \label{main:isometry}
\end{enumerate}

\end{thm}

\begin{table}
\small{
\begin{tabular}{c||c|c|c|c}
Reference & $X$ & $Y$   & $f$ & Conclusion about $f$  \\ \hline \hline
Question\ \ref{quest:quasi-preserve_a.e.}  & - & - &  Lip.\ & BLD on a.e.\ curve? \\ \hline
Thm.\ \ref{thm:main} \ref{main:length}  & Reciprocal & - &  ($1$-)Lip.\ & ($1$-)BLD on a.e.\ curve \\ \hline
Example \ref{example:collapse} &Riemannian & - &  $1$-Lip. & \textit{Not} homeomorphic\\ \hline 
Thm.\ \ref{thm:main} \ref{main:injective}& -& Reciprocal   & ($1$-)Lip. & ($1$-)QC homeom., \\ & & & & ($1$-)BLD on a.e.\ curve\\ \hline
Example \ref{example:weight} & Riemannian & Reciprocal & 1-Lip. & \textit{Not} BLD \\  \hline 
Thm.\ \ref{thm:main}\ref{main:bld} &- & Upper regular & Lip. & QC homeom., BLD\\\hline 
Example \ref{example:sphere} & Riemannian & Upper regular & $1$-Lip.\ & \textit{Not} isometric\\\hline
Thm.\ \ref{thm:main}\ref{main:isometry} &- & Riemannian & $1$-Lip. & Isometry
\end{tabular}
}
\caption{The conclusions of Theorem \ref{thm:main}. In all cases $f$ is assumed to be area-preserving. }\label{table}
\end{table}

We were neither able to show that part \ref{main:length} holds without the assumption that $X$ is reciprocal, nor were we able to find a counterexample. This raises the following question.

\begin{question}\label{quest:quasi-preserve_a.e.}
Suppose that $X,Y$ are metric surfaces of locally finite Hausdorff $2$-measure. If $f\colon X\to Y$ is an area-preserving and Lipschitz map, does it quasi-preserve the length of a.e.\ path in $X$?
\end{question}

In Section \ref{section:examples} we present examples illustrating the optimality of Theorem \ref{thm:main}. We first note that area-preserving and $1$-Lipschitz maps are not injective in general without any assumptions on $Y$; a sufficient condition is the reciprocity of $Y$ in part \ref{main:injective}. Moreover, one cannot expect in  part \ref{main:injective} that the length of \textit{all} curves (rather than a.e.\ curve) is quasi-preserved; a sufficient condition is upper Ahlfors $2$-regularity of $Y$ as in  \ref{main:bld}.  Finally, in part \ref{main:bld} one cannot expect a $1$-Lipschitz map $f$ to be an isometry without further assumptions on $Y$, such as smoothness, as in \ref{main:isometry}; this has already been illustrated in Example \ref{example:sphere}.

\subsection{Coarea inequality}
The proof of Theorem \ref{thm:main} relies on a coarea inequality for continuous Sobolev functions on metric surfaces. The following result is an improvement of the coarea inequality for \textit{monotone} Sobolev functions that was established recently in \cite{EIR22}; here monotonicity means that the maximum and minimum of a function on a precompact open set are attained at the boundary. We direct the reader to \cite{EIR22} for further background on the coarea inequality in metric spaces. 

\begin{thm}\label{theorem:coarea:introduction}
    Let $X$ be a metric surface of locally finite Hausdorff $2$-measure and $u\colon X\to \R$ be a continuous function with a $2$-weak upper gradient $\rho_u\in L^2_{\loc}(X)$. 
    \begin{enumerate}[label=\normalfont(\arabic*)]
        \item If $\mathcal A_u$ denotes the union of all non-degenerate components of the level sets $u^{-1}(t)$, $t\in \R$, of $u$, then $\mathcal A_u$ is a Borel set. \label{ca:i}
        \item For every Borel function $g\colon X\to [0,\infty]$ we have
        $$\int\displaylimits^* \int_{u^{-1}(t)\cap \mathcal A_u}g\, d\mathcal H^1\, dt \leq \frac{4}{\pi} \int g\rho_u\, d\mathcal H^2.$$\label{ca:ii}
        \item If, in addition, $u$ is Lipschitz, then for every Borel function $g\colon X\to [0,\infty]$ we have
        $$\int\displaylimits^* \int_{u^{-1}(t)}g\, d\mathcal H^1\, dt \leq \frac{4}{\pi} \int g\cdot (\rho_u \chi_{\mathcal A_u}+ \lip(u)\chi_{X\setminus \mathcal A_u})\, d\mathcal H^2.$$\label{ca:iii}
    \end{enumerate}
\end{thm}
Here $\lip(u)$ denotes the pointwise Lipschitz constant of a Lipschitz function $u\colon X\to \R$, defined by
$$\lip(u)(x)=\limsup_{y\to x}\frac{|u(y)-u(x)|}{d(x,y)}. $$
Also, $\int^*$ denotes the upper integral, which is equal to the Lebesgue integral for measurable functions. The main result of \cite{EIR22} (for $p\geq 2$) states that \ref{ca:ii} holds with the additional assumption that $u$ is monotone and with $u^{-1}(t)$ in place of $u^{-1}(t)\cap \mathcal A_u$. Since the level sets of monotone functions are always non-degenerate (see e.g.\ \cite{Nt:monotone}*{Corollary 2.8}), we see that $\mathcal A_u=X$ when $u$ is monotone; hence our theorem implies the main result of \cite{EIR22} for $p\geq 2$.  Moreover, without the monotonicity assumption, we note that part \ref{ca:ii} is optimal and does not hold for the full level sets $u^{-1}(t)$ if we do not restrict to $\mathcal A_u$, even if $u$ is Lipschitz. A relevant example is provided in \cite{EIR22}*{Section 5}. 

The proof of Theorem \ref{theorem:coarea:introduction} relies on recent developments in the theory of uniformization of metric surfaces. Specifically, we use a result of Romney and the second-named author \cite{NR22}, which states that every metric surface of locally finite Hausdorff $2$-measure admits a weakly quasiconformal parametrization by a Riemannian surface of the same topological type.

\subsection*{Acknowledgments}
The paper was written while the first-named author was visiting Stony Brook University. She wishes to thank the department for their hospitality. The authors would like to thank Stefan Wenger for his comments on a draft of the paper.

\section{Preliminaries}

\subsection{Hausdorff measures}
For a metric space $X$ and $s>0$, the \textit{Hausdorff $s$-measure} of a set $A \subset X$ is defined by
$$\mathcal{H}^s(A) = \lim_{\delta\to 0} \mathcal H^s_{\delta}(A),\,\, \textrm{where}\,\,\,\, \mathcal H^s_\delta(A) =\inf \left\{ \sum_{j=1}^\infty \frac{\omega_s}{2^s} \diam(A_j)^s\right\} $$
and the infimum is taken over all collections of sets $\{A_j\}_{j=1}^\infty$ such that $A \subset \bigcup_{j=1}^\infty A_j$ and $\diam(A_j) < \delta$ for each $j$. Here $\omega_s$ is a positive normalization constant, chosen so that the Hausdorff $n$-measure coincides with Lebesgue measure in $\R^n$. Note that $\omega_1=2$ and $\omega_2=\pi$.  If we need to emphasize the metric $d$ being used for the Hausdorff $s$-measure, we write $\mathcal{H}_{d}^s$ instead of $\mathcal{H}^s$.

We  state the  coarea inequality for Lipschitz functions and the classical coarea formula for Sobolev functions. 
\begin{thm}[Coarea inequality and formula]\label{theorem:coarea:classical}
    Let $X$ be a metric space, $u\colon X\to \R$ be a continuous function, and $g\colon X\to [0,\infty]$ be a Borel function.
    \begin{enumerate}[label=\normalfont(\arabic*)]
        \item If $u$ is Lipschitz, then for $K=4/\pi$ we have
    \begin{align*}
        \int\displaylimits^* \int_{u^{-1}(t)} g\, d\mathcal H^1dt \leq K \int_X g\cdot \lip(u)   \, d\mathcal H^2.
    \end{align*}
   If $X$ is a Riemannian surface, we may take $K=1$.  \label{coarea:1}
        \item If $X$ is an open subset of $\R^2$ and $u\in W^{1,1}_{\loc}(X)$, then
    \begin{align*}
        \int \int_{u^{-1}(t)}g\, d\mathcal H^1dt= \int_X g\cdot |\nabla u|\, d\mathcal H^2.
    \end{align*}\label{coarea:2}
    \end{enumerate}
\end{thm}

Part \ref{coarea:1} is a consequence of \cite{Fed69}*{Theorem 2.10.25} for general metric spaces $X$ and of \cite{Fed:coarea}*{Theorem 3.1} for Riemannian manifolds with $K=1$. Part \ref{coarea:2} is stated in \cite{MalySwansonZiemer:Coarea} and attributed to Federer. See also \cite{EsmayliHajlasz:coarea} for a more general statement than \ref{coarea:1} and \cite{EIR22}*{Lemma 5.2}.

\subsection{Topological preliminaries}

Let $X$ be a metric space. A \textit{path} or \textit{curve} is a continuous map $\gamma\colon [a,b]\to X$. The \textit{trace} of $\gamma$ is the set $|\gamma|=\gamma([a,b])$. We say that $\gamma$ is a \textit{Jordan arc} if $\gamma$ is injective. Here we allow the possibility $a=b$, in which case $\gamma$ is a \textit{degenerate} Jordan arc.  We say that $\gamma$ is a \textit{Jordan curve} if $\gamma|_{[a,b)}$ is injective and $\gamma(a)=\gamma(b)$. We also say that a {set} $K\subset X$ is a Jordan arc (resp.\ Jordan curve) if there exists a Jordan arc (resp.\ Jordan curve) $\gamma$ with $|\gamma|=K$. The \textit{length} of a curve $\gamma$ is its total variation and is denoted by $\ell(\gamma)$. A \textit{continuum} is a compact and connected metric space. A \textit{Peano continuum} is a locally connected continuum.

\begin{lemma}\label{lemma:triod}
    Let $\{K_i\}_{i\in I}$ be a collection of pairwise disjoint Peano continua in $\R^2$. Then, with the exception of countably many $i\in I$, each $K_i$ is a Jordan arc or a Jordan curve. 
\end{lemma}
    \begin{proof}
        A \textit{triod} is the union of three non-degenerate Jordan arcs that have a common endpoint, the \textit{junction point}, but are otherwise disjoint. A theorem of Moore \cite{Moore:triods} (see also \cite{Pommerenke:conformal}*{Proposition 2.18}) states that there is no uncountable collection of pairwise disjoint triods in the plane. On the other hand, if a Peano continuum is not a Jordan arc or Jordan curve, then it contains a triod \cite{Nt:monotone}*{Lemma 2.4}. This completes the proof. 
    \end{proof}

\begin{lemma}\label{lemma:peano_hausdorff}
    Let $K$ be a continuum with $\mathcal H^1(K)<\infty$. Then $K$ is a Peano continuum.  
\end{lemma}
\begin{proof}
       If $\mathcal H^1(K)<\infty$, a result of Eilenberg--Harrold \cite{EilenbergHarrold:Curves}*{Theorem 2} states that there exists a continuous and surjective mapping $\gamma\colon  [0,1] \to K$ (with $\ell(\gamma)\leq 2\mathcal H^1(K)-\diam(K)$). By the Hahn--Mazurkiewicz theorem \cite{Willard:topology}*{Theorem 31.5}, Peano continua are characterized as continuous images of the unit interval.  
\end{proof}

\begin{lemma}[\cite{BuragoBuragoIvanov:metric}*{Theorem 2.6.2}]\label{lemma:length_hausdorff}
    Let $X$ be a metric space and let $\gamma \colon [a,b]\to X$ be a curve. Then $\ell(\gamma)\geq \mathcal H^1(|\gamma|)$. Moreover, if $\gamma$ is a Jordan arc or Jordan curve, then $\ell (\gamma)=\mathcal H^1( |\gamma|)$. 
\end{lemma}

We state a consequence of Lemmas \ref{lemma:triod}, \ref{lemma:peano_hausdorff}, and \ref{lemma:length_hausdorff}, and of the existence of arclength parametrizations of rectifiable curves \cite{HKST:15}*{Section 5.1}.  

\begin{corollary}\label{corollary:parametrizations}
    Let $X$ be a metric space homeomorphic to a subset of $\R^2$. Let $\{K_i\}_{i\in I}$ be a collection of pairwise disjoint continua in $X$ with $\mathcal H^1(K_i)<\infty$ for each $i\in I$. Then, with the exception of countably many $i\in I$, each $K_i$ is a Jordan arc or a Jordan curve and there exists a Lipschitz parametrization $\gamma\colon [a_i,b_i]\to K_i$ that is injective in $[a_i,b_i)$. 
\end{corollary}

\begin{lemma}\label{lemma:separation_component}
    Let $X$ be a topological space homeomorphic to $\mathbb S^2$ or to a closed disk. Let $K\subset X$ be a compact set separating two points $a,b\in X$. Then there exists a connected component of $K$ that also separates $a$ and $b$.  
\end{lemma}

In $\mathbb S^2$ this is a consequence of \cite{Wilder:topology}*{Lemma II.5.20, p.~61}. For topological disks the conclusion follows from \cite{LWintrinsic}*{Lemma 7.1}.

Throughout the paper $\inter(X)$ denotes the manifold interior of a surface $X$. The topological interior of a set $A$ in a topological space is denoted by $\inter_{\top}(A)$. Similar notation is adopted for the notion of boundary.

\subsection{Metric Sobolev spaces}\label{section:sobolev}
Let $X$ be a metric space and $\Gamma$ be a family of curves in $X$. A Borel function $\rho\colon X \to [0,\infty]$ is \textit{admissible} for $\Gamma$ if $\int_{\gamma}\rho\, ds\geq 1$
for all rectifiable paths $\gamma\in \Gamma$. We define the \textit{$2$-modulus} of $\Gamma$ as 
$$\mod \Gamma = \inf_\rho \int_X \rho^2 \, d\mathcal H^2,$$
where the infimum is taken over all admissible functions $\rho$ for $\Gamma$. By convention, $\mod \Gamma = \infty$ if there are no admissible functions for $\Gamma$. Observe that we consider $X$ to be equipped with the Hausdorff 2-measure. This definition may be generalized by allowing for an exponent different from $2$ or a different measure, though this generality is not needed for this paper.

Let $h\colon X\to Y$ be a map between metric spaces. We say that a Borel function $g\colon X\to [0,\infty]$ is an \textit{upper gradient} of $h$ if 
\begin{align}\label{ineq:upper_gradient}
    d_Y(h(x),h(y)) \leq \int_{\gamma} g \, ds
\end{align}
for all $x,y\in X$ and every rectifiable path $\gamma$ in $X$ joining $x$ and $y$. This is called the \textit{upper gradient inequality}. If, instead the above inequality holds for all curves $\gamma$ outside a curve family of $2$-modulus zero, then we say that $g$ is a \textit{($2$-)weak upper gradient} of $h$. In this case, there exists a curve family $\Gamma_0$ with $\mod \Gamma_0=0$ such that all paths outside $\Gamma_0$ and all subpaths of such paths satisfy the upper gradient inequality.

{We equip the space $X$ with the Hausdorff $2$-measure $\mathcal{H}^2$.} Let $L^2(X)$ denote the space of $2$-integrable Borel functions from $X$ to the extended real line $\widehat{\mathbb{R}}$, where two functions are identified if they agree $\mathcal{H}^2$-almost everywhere. The Sobolev space $N^{1,2}(X,Y)$ is defined as the space of Borel maps $h \colon X \to Y$ with a $2$-weak upper gradient $g$ in $L^2(X)$ such that the function $x \mapsto d_Y(y,h(x))$ is in $L^2(X)$ for some $y \in Y$. If $Y=\R$, we simply write $N^{1,2}(X)$. The spaces $L_{\loc}^2(X)$ and $N_{\loc}^{1,2}(X, Y)$ are defined in the obvious manner. Each map $h\in N_{\loc}^{1,2}(X,Y)$ has a \textit{minimal} $2$-weak upper gradient $g_h$, in the sense that for any other $2$-weak upper gradient $g$ we have $g_h\leq g$ a.e. See the monograph \cite{HKST:15} for background on metric Sobolev spaces.

We state a consequence of the coarea inequality for Lipschitz functions. 
\begin{lemma}[\cite{EIR22}*{Lemma 2.13}]\label{lemma:avoid}
    Let $X$ be a metric surface of finite Hausdorff $2$-measure and $u\colon X\to \R$ be a Lipschitz function. If $\Gamma_0$ is a curve family in $X$ with $\mod \Gamma_0=0$, then for a.e.\ $t\in \R$, every Lipschitz curve $\gamma\colon[a,b]\to u^{-1}(t)$ that is injective on $[a,b)$ lies outside $\Gamma_0$.
\end{lemma}

\subsection{Quasiconformal maps}

Let $X,Y$ be metric surfaces of locally finite Hausdorff $2$-measure. Recall that a homeomorphism $h\colon X\to Y$ is {quasiconformal} if there exists $K\geq 1$ such that 
$$K^{-1} \mod\Gamma\leq \mod h(\Gamma)\leq K\mod\Gamma$$
for every curve family $\Gamma$ in $X$. A continuous map between topological spaces is \textit{cell-like} if the preimage of each point is a continuum that is contractible in each of its open neighborhoods. A continuous, surjective, proper, and cell-like map $h\colon X\to Y$ is \emph{weakly quasiconformal} if  there exists $K>0$ such that for every curve family $\Gamma$ in $X$ we have
$$\mod \Gamma\leq K \mod h(\Gamma).$$
In this case, we say that $h$ is weakly $K$-quasiconformal. 

If $X$ and $Y$ are compact surfaces that are homeomorphic to each other, then we may replace cell-likeness with the weaker requirement that $h$ is monotone; that is, the preimage of every {point} is a continuum. In that case, continuous, surjective, and monotone maps from $X$ to $Y$ coincide with uniform limits of homeomorphisms; see \cite{NR22}*{Theorem 6.3} and the references therein. Alternatively, if $X,Y$ have empty boundary, then continuous, proper, and cell-like maps from $X$ to $Y$ also coincide with uniform limits of homeomorphisms, see \cite{Dav86}*{Corollary 25.1A}.

We note that a weakly $K$-quasiconformal map between planar domains is a $K$-quasiconformal homeomorphism. Indeed, by \cite{NR:21}*{Theorem 7.4}, such a map is a homeomorphism. Also, note that a quasiconformal homeomorphism between planar domains is a priori required to satisfy only one modulus inequality, as in the definition of a weakly quasiconformal map; see \cite{LehtoVirtanen:quasiconformal}*{Section I.3}. 

The next  theorem of Williams (\cite{Wil:12}*{Theorem 1.1 and Corollary 3.9}) relates the above definitions of quasiconformality with the ``analytic" definition that relies on upper gradients; see also \cite{NR:21}*{Section 2.4}.

\begin{thm}[Definitions of quasiconformality]\label{theorem:definitions_qc}
Let $X,Y$ be metric surfaces of locally finite Hausdorff $2$-measure, $h\colon X\to Y$ be a continuous map, and $K>0$. The following are equivalent.
\begin{enumerate}[label=\normalfont(\roman*)]
    \item\label{def:i} $h\in N^{1,2}_{\loc}(X,Y)$ and there exists a $2$-weak upper gradient $g$ of $h$ such that for every Borel set $E\subset Y$ we have
    $$\int_{h^{-1}(E)} g^2\, d\mathcal H^2 \leq K \mathcal H^2(E).$$
   \item\label{def:i'}Each point of $X$ has a neighborhood $U$ such that $h|_U\in N^{1,2}(U,Y)$ and there exists a $2$-weak upper gradient $g_U$ of $h|_U$ such that for every Borel set $E\subset Y$ we have
    $$\int_{(h|_{U})^{-1}(E)} g_U^2\, d\mathcal H^2 \leq K \mathcal H^2(E).$$
    \item\label{def:ii} For every curve family $\Gamma$ in $X$ we have
    $$\mod \Gamma \leq K\mod h(\Gamma).$$
\end{enumerate}
\end{thm}

\begin{thm}[\cite{NR:21}*{Theorem 7.1 and Remark 7.2}]\label{theorem:radon-nikodym}
    Let $X,Y$ be metric surfaces of locally finite Hausdorff $2$-measure and $h\colon X\to Y$ be a weakly $K$-quasiconformal map for some $K>0$. 
    \begin{enumerate}[label=\normalfont(\arabic*)]
        \item The set function $\nu(E)=\mathcal H^2(h(E))$ is a locally finite Borel measure on $X$. Moreover, for a.e.\ $x\in X$ we have 
        $$g_h(x)^2\leq KJ_h(x), \quad \textrm{where}\,\,\, J_h= \frac{d\nu}{d\mathcal H^2}.$$ 
        \item $N(h,y)=1$ for a.e.\ $y\in Y$.
    \end{enumerate}
\end{thm}

Here, $N(h,y)$ denotes the number of preimages of $y$ under $h$.

\subsection{Reciprocal surfaces}\label{section:reciprocal}
 
Let $X$ be a metric surface of locally finite Hausdorff $2$-measure. For a set $G\subset X$ and disjoint sets $E,F\subset G$ we define $\Gamma(E,F;G)$ to be the family of curves in $G$ joining $E$ and $F$. A \textit{quadrilateral} in $X$ is a closed Jordan region $Q$ together with a partition of $\partial Q$ into four non-overlapping edges $\zeta_1,\zeta_2,\zeta_3,\zeta_4\subset \partial Q$ in cyclic order. When we refer to a quadrilateral $Q$, it will be implicitly understood that there exists such a marking on its boundary. We define $\Gamma(Q)=\Gamma(\zeta_1,\zeta_3;Q)$  and $\Gamma^*(Q) =\Gamma(\zeta_2,\zeta_4;Q)$. According to the definition of Rajala \cite{Raj:17}, the metric surface $X$ is {reciprocal} if there exist constants $\kappa,\kappa'\geq 1$ such that 
\begin{align}\label{reciprocality:12}
    \kappa^{-1}\leq \mod \Gamma(Q) \cdot \mod\Gamma^*(Q) \leq \kappa' \quad \textrm{for each quadrilateral $Q\subset X$}
\end{align}
and 
\begin{align}\label{reciprocality:3}
        &\lim_{r\to 0} \mod \Gamma( \bar B(a,r), X\setminus B(a,R);X )=0 \quad \textrm{for each ball $B(a,R)$.} 
\end{align}
By work of Rajala and Romney \cite{RR:19} it is now known that the lower bound in \eqref{reciprocality:12} is always satisfied for some uniform constant $\kappa$. In fact, the optimal constant was shown to be $\kappa=(4/\pi)^2$ \cite{EBC:21}.  Moreover, \eqref{reciprocality:3} follows from the upper bound in \eqref{reciprocality:12}, as was shown by Romney and the second-named author \cite{NR22}. Therefore, we may only require the upper inequality of \eqref{reciprocality:12} in the definition of a reciprocal surface.

Rajala \cite{Raj:17} proved that a metric surface $X$ of locally finite Hausdorff $2$-measure that is homeomorphic to $\R^2$ is $2$-quasiconformally equivalent to an open subset of $\R^2$ if and only if $X$ is reciprocal. This result was generalized to all metric surfaces (with or without boundary) of locally finite Hausdorff $2$-measure, where $\R^2$ is replaced with a Riemannian surface \cites{Iko:19,NR22}. 

More generally, it was shown in \cite{NR22} that any metric surface of locally finite Hausdorff $2$-measure admits a weakly quasiconformal parametrization by a Riemannian surface of the same topological type. The following special case is sufficient for our purposes.

\begin{thm}[\cite{NR22}*{Theorem 1.2}]\label{theorem:wqc}
    Let $X$ be a metric surface of finite Hausdorff $2$-measure that is homeomorphic to a topological closed disk. Then there exists a weakly $(4/\pi)$-quasiconformal map from $\bar \D$ onto $X$. 
\end{thm}

Here $\D$ denotes the open unit disk in the plane. We show that weakly quasiconformal maps can be upgraded to quasiconformal homeomorphisms under certain conditions.

\begin{lemma}\label{lemma:weak_qc_upgrade}
    Let $X,Y$ be metric surfaces without boundary and with locally finite Hausdorff $2$-measure such that $Y$ is reciprocal. Then every weakly quasiconformal map $f\colon X\to Y$ is a quasiconformal homeomorphism, quantitatively. 
\end{lemma}
\begin{proof}
    Let $f\colon X\to Y$ be a weakly $K$-quasiconformal map for some $K>0$. Since $Y$ is reciprocal, condition \eqref{reciprocality:3} implies that the modulus of the family of non-constant curves passing through any point of $Y$ is zero. By \cite{NR:21}*{Theorem 7.4} we conclude that $f$ is a homeomorphism. Now, the reciprocity of $Y$ implies that the upper bound in \eqref{reciprocality:12} is satisfied for $X$ as well. Therefore, $X$ is reciprocal. 
    
    Consider a domain $V'\subset Y$ that is homeomorphic to $\R^2$. By Rajala's theorem, there exists a $2$-quasiconformal homeomorphism $\phi$ from $V'$ onto a domain $V\subset \R^2$. The set $U'=f^{-1}(V')$ is homeomorphic to $\R^2$, so by Rajala's theorem there exists a $2$-quasiconformal homeomorphism $\psi$ from $U'$ onto a domain $U\subset \R^2$. The composition $g=\phi\circ f\circ \psi^{-1}$ is a weakly $4K$-quasiconformal map from $U$ onto $V$. Since the domains are planar, $g$ is a $4K$-quasiconformal homeomorphism. Therefore, $f$ is a $16K$-quasiconformal homeomorphism from $U'$ onto $V'$. By Theorem \ref{theorem:definitions_qc}, quasiconformality is a local condition, so $f\colon X\to Y$ is $16K$-quasiconformal. 
\end{proof}

\subsection{Metric differentiability}
Throughout the section we let $U\subset\R^2$ be a domain and $Y$ be a metric space. We say that a map $h\colon U\to Y$ is \emph{approximately metrically differentiable} at a point $x\in U$ if there exists a seminorm $N_x$ on $\R^2$ for which
    $$\ap\lim_{y\to x}\frac{d(h(y),h(x))-N_x(y-x)}{y-x}=0.$$
Here, $\ap\lim$ denotes the approximate limit as defined in \cite{EG92}*{Section 1.7.2}. In this case, the seminorm $N_x$ is unique, is denoted by $\ap\md h_x$, and we call it the \emph{approximate metric derivative} of $h$ at $x$. 

\begin{proposition}[\cite{LWarea}*{Proposition 4.3}]\label{prop:decomp-apmd}
    If $h\in N^{1,2}(U,Y)$ then there exist countably many pairwise disjoint compact sets $K_i\subset U$, $i\in\N$, such that $\mathcal H^2(U\setminus\bigcup_{i\in\N}K_i)=0$ with the following property. For every $i\in\N$ and every $\varepsilon>0$ there exists $r_i(\varepsilon)>0$ such that $h$ is approximately metrically differentiable at every $x\in K_i$ and
    \begin{align*}
        |d(h(x),h(x+v))-\ap\md h_x(v)|\leq\varepsilon|v|
    \end{align*}
    for all $x\in K_i$ and all $v\in\R^2$ with $|v|\leq r_i(\varepsilon)$ and $x+v\in K_i$.
\end{proposition}

In particular, every map $h\in N^{1,2}(U,Y)$ is approximately metrically differentiable at a.e.\ $x\in U$.

\begin{lemma}[\cite{LWintrinsic}*{Lemma 3.1}]\label{lemma:length-apmd}
    If $h\in N^{1,2}(U,Y)$ then
    $$\ell(h\circ\gamma)=\int_a^b \ap\md h_{\gamma(t)}(\dot\gamma(t))\,dt$$ for every curve $\gamma\colon [a,b]\to U$ parametrized by arclength outside a family $\Gamma_0$ with $\mod\Gamma_0=0$. 
\end{lemma}

\begin{lemma}\label{lemma:apmd_ug}
    If $h\in N^{1,2}(U,X)$ then the function $L\colon U\to[0,\infty]$ defined by $L(x)=\max\{\ap\md h_x(v):|v|=1\}$ is a representative of the minimal 2-weak upper gradient of $h$.
\end{lemma}
\begin{proof}
    It is an immediate consequence of Lemma \ref{lemma:length-apmd} that $L$ is a $2$-weak upper gradient of $h$.  It remains to show that if $g$ is an upper gradient of $h$ in $L^2(U)$, then $L(x)\leq g(x)$ for a.e.\ $x\in U$; this will imply that the same conclusion is true for the minimal $2$-weak upper gradient. Let $g\in L^2(U)$ be an upper gradient of $h$. It can be deduced from Fubini's theorem that for each $v\in \mathbb S^1$ and for a.e.\ $x\in U$ we have
    \begin{align}\label{eq:g(x)}
        g(x)=\lim_{\delta\to 0}\frac{1}{\delta}\int_0^{\delta}g(x+t v)\,dt = \lim_{\delta\to 0} \frac{1}{\delta}\int_{\gamma_v|_{[0,\delta]}} g\, ds,
    \end{align}
    where $\gamma_{v}\colon [0,1] \to \R^2$ is the curve $\gamma_{v}(t)=x+tv$. Consider a set $K_i$ as in Proposition \ref{prop:decomp-apmd}. An application of Fubini's theorem shows that for each $v\in \mathbb S^1$ and for a.e.\ $x\in K_i$ we have $x+\delta v\in K_i$ for arbitrarily small values of $\delta>0$. Let $\varepsilon>0$, $v\in \mathbb S^1$, and $x\in K_i$ such that \eqref{eq:g(x)} is true and $x+\delta_n v\in K_i$ for a sequence $\delta_n\to 0$. By Proposition \ref{prop:decomp-apmd}, whenever $|\delta_n v|\leq r_i(\varepsilon)$, we have
   \begin{align*}
       \ap\md h_x(v)&\leq\frac{1}{\delta_n}d(h(x),h(x+\delta_n v))+\varepsilon |v|  \leq\frac{1}{\delta_n}\int_{\gamma_{v}|_{[0,\delta_n]}}g+\varepsilon.
   \end{align*}
   We let $n\to \infty$ and then $\varepsilon\to 0$ to obtain $\ap\md h_x(v)\leq g(x)$. Since this is true for every $v\in \mathbb S^1$, we obtain $L(x)\leq g(x)$ for a.e.\ $x\in K_i$. The sets $K_i$, $i\in \N$, cover $U$ up to a set of measure zero, so the conclusion follows. 
\end{proof}

If $\ap\md h_x$ is a norm, let $B_x=\{y\in\R^2: \ap\md h_x(y)\leq1\}$ be the closed unit ball in $(\R^2, \ap\md h_x)$. The \emph{Jacobian} of $\ap\md h_x$ is defined to be $J(\ap\md h_x)={\pi}/{|B_x|}$, where $|B_x|$ is the Lebesgue measure of $B_x$. As a planar symmetric convex body, $B_x$ contains a unique ellipse of maximal area $E_x$, called the \emph{John ellipse} of $B_x$; see \cite{Bal97}*{Theorem 3.1}. When $\ap\md h_x$ is not a norm, the closed unit ball $B_x$ has infinite area and we define $J(\ap\md h_x)=0$. 

\begin{thm}[Area formula]\label{thm:area-formula}
    If $h\in N^{1,2}(U,Y)$, then there exists a set $G_0\subset U$ with $\mathcal H^2(G_0)=0$ such that for every measurable set $A\subset U\setminus G_0$ we have
    \begin{align*}
       \int_A J(\ap\md h_x) \,d\mathcal H^2=\int_Y N(h|_A,y)\,d\mathcal H^2.
    \end{align*}
\end{thm}
\begin{proof}
    It is a consequence of \cite{HKST:15}*{Theorem 8.1.49} that $U$ can be covered up to a set of measure zero by countably many disjoint measurable sets $G_j$, $j\in \N$, such that $h|_{G_j}$ is Lipschitz. This implies that outside a set of measure zero $G_0\subset U$, $h$ satisfies Lusin's condition (N). The statement now follows from \cite{Kar07}*{Theorem 3.2}.
\end{proof}

\begin{lemma}\label{lemma:wqc_derivative}
Let $Y$ be a metric surface of locally finite Hausdorff $2$-measure and $h\colon U\to Y$ be a weakly $K$-quasiconformal map for some $K>0$. Then
 $$J(\ap\md h_x)\leq  \max\{(\ap\md h_x(v))^2:|v|=1\} \leq KJ(\ap\md h_x)$$
 for a.e.\ $x\in U$. In particular, for a.e.\ $x\in U$ we have  $J(\ap\md h_x)=0$ if and only if $\ap \md h_x\equiv 0$.
\end{lemma}
\begin{proof}
    By Theorem \ref{theorem:definitions_qc}, $h\in N^{1,2}_{\loc}(U,Y)$, so $h$ is approximately metrically differentiable at a.e.\ $x\in U$. We set $N_x=\ap\md h_x$ and $J_x=J(\ap\md h_x)$ for a.e.\ $x\in U$.  By Lemma \ref{lemma:apmd_ug}, the quantity $L_x=\max\{N_x(v):|v|=1\}$ is a representative of the minimal $2$-weak upper gradient of $h$, so $L_x=g_h(x)$ for a.e.\ $x\in U$. By the area formula of Theorem \ref{thm:area-formula}, there exists a set $G_0\subset U$ of measure zero such that for each measurable set $A\subset U\setminus G_0$ we have
    $$\int_A J_x= \int_{Y} N(h|_A,y)\, d\mathcal H^2= \mathcal H^2(h(A)),$$
    where the latter equality follows from Theorem \ref{theorem:radon-nikodym}. This implies that $J_x$ is the Radon--Nikodym derivative of the measure $A\mapsto \mathcal H^2(h(A))$, so $J_x=J_h(x)$ for a.e.\ $x\in U$, again by Theorem \ref{theorem:radon-nikodym}.  Finally, since $g_h(x)^2\leq KJ_h(x)$, we conclude that $L_x^2\leq KJ_x$ for a.e.\ $x\in U$. The inequality $J_x\leq L_x^2$ follows by the fact that the unit ball $B_x=\{y\in \R^2: N_x(y)\leq 1\}$ contains a Euclidean ball of radius $1/L_x$. 
\end{proof}

\begin{remark}
It is a consequence of Lemma \ref{lemma:wqc_derivative} that if $f$ is a weakly $K$-quasiconformal map from a planar (or Riemannian) domain $U$ onto a metric surface $Y$, then we necessarily have $K\geq 1$. It is unclear how to show this for maps between arbitrary metric surfaces. 
\end{remark}

\section{Proof of main theorem}This section is devoted to the proof of Theorem \ref{thm:main}. Throughout the section we assume that $X,Y$ are metric surfaces without boundary and with locally finite Hausdorff $2$-measure. 

\subsection{Preservation of length}
In this section we establish Theorem \ref{thm:main} \ref{main:length}. 

\begin{lemma}\label{lemma:weak_qc}
    Let $f\colon X\to Y$ be a map that is area-preserving and $L$-Lipschitz for some $L>0$. Then $\mod \Gamma\leq L^2\mod f(\Gamma)$ for each curve family $\Gamma$ in $X$. 
\end{lemma}
\begin{proof}
    Since $f$ is $L$-Lipschitz, the constant function $L$ is an upper gradient of $f$. Moreover, for every Borel set $A\subset Y$ we have
    $$\int_{f^{-1}(A)} L^2 \, d\mathcal H^2= L^2 \mathcal H^2(f^{-1}(A))= L^2 \mathcal H^2( f(f^{-1}(A))) \leq L^2\mathcal H^2(A).$$
    The conclusion now follows from Theorem \ref{theorem:definitions_qc}. 
\end{proof}

\begin{lemma}\label{lemma:multiplicity}
Let $f\colon X\to Y$ be a map that is area-preserving and continuous. Then $N(f,y)=1$ for a.e.\ $y\in f(X)$. 
\end{lemma}
\begin{proof}
For each Borel set $A\subset X$ the set $f(A)$ is analytic \cite{Kechris}*{Proposition 14.4} and thus $\mathcal H^2$-measurable \cite{Kechris}*{Theorem 29.7}. Define $\zeta(S)= \mathcal H^2(f(S))$, where $S\subset X$ is a Borel set. By assumption, $\zeta(S)=\mathcal H^2(S)$. The measure on $X$ resulting by Carath\'eodory's construction from $\zeta$ is precisely $\mathcal H^2$. By \cite{Fed69}*{Theorem 2.10.10}, for each Borel set $A\subset X$ we have
$$ \mathcal H^2(A)=\int_Y N(f|_A , y) \, d\mathcal H^2.$$
In particular, since $f$ is area-preserving we have
\begin{align*}
    \mathcal H^2 (A) =\int_{f(A)} N(f|_A , y)\, d\mathcal H^2\geq \mathcal H^2(f(A)) =\mathcal H^2(A).
\end{align*}
If $\mathcal H^2(A)<\infty$, we conclude that $N(f|_A,y)=1$ for a.e.\ $y\in f(A)$. Since $X$ has $\sigma$-finite Hausdorff $2$-measure,  we have $N(f,y)=1$ for a.e.\ $y\in f(X)$.
\end{proof}

\begin{lemma}\label{lemma:length}
Let $U\subset \R^2$ be a domain and $\phi\colon U\to X$ be a weakly quasiconformal map. Let $f\colon X\to Y$ be a map that is area-preserving and $L$-Lipschitz for some $L>0$. Then there exists a constant $C(L)>0$ such that 
\begin{align*}
    C(L)\ell(\phi\circ \beta)\leq \ell(f\circ \phi\circ \beta)\leq L\ell(\phi\circ \beta)
\end{align*}
for all curves $\beta$ in $U$ outside a curve family $\Gamma_0$ with $\mod \Gamma_0=0$. Moreover, if $L=1$, then we can choose $C(1)=1$.
\end{lemma}

\begin{proof}
    By Lemma \ref{lemma:weak_qc} and the weak quasiconformality of $\phi$, there exists a constant $K \geq 1$ such that for each curve family $\Gamma$ in $U$ we have
    $$\mod \Gamma\leq K\mod f(\phi(\Gamma)).$$
    By Theorem \ref{theorem:definitions_qc}, $f\circ \phi\in N_{\loc}^{1,2}(U,Y)$ and $\phi\in N_{\loc}^{1,2}(U ,X)$. In particular, both maps are approximately metrically differentiable almost everywhere. 

    Set $N_x=\ap\md \phi_x$ and $\widetilde N_x=\ap\md (f\circ\phi)_x$ for a.e.\ $x\in U$. We use the notation $B_x$, $\widetilde B_x$ for the corresponding unit balls, and $J_x$, $\widetilde J_x$ for the corresponding Jacobians. By Lemma \ref{lemma:length-apmd} we have
    \begin{align}\label{length:gamma}
        \ell(\phi\circ\beta)=\int_a^bN_{\beta(t)}(\dot\beta(t))\,dt
    \end{align}
    for every curve $\beta\colon [a,b]\to U$ parametrized by arclength outside a family $\Gamma_1$ with $\mod\Gamma_1=0$. Analogously, we get 
    \begin{align}\label{length:fgamma}
        \ell((f\circ\phi)\circ\beta)=\int_a^b\widetilde{N}_{\beta(t)}(\dot\beta(t))\,dt
    \end{align}
    for every curve $\beta\colon [a,b]\to U$ parametrized by arclength outside a family $\Gamma_2$ with $\mod\Gamma_2=0$. 
    
    Next, we claim that for a.e.\ $x\in U$ and all\ $v\in\R^2$ we have, 
    $$C(L)N_x(v)\leq\widetilde{N}_x(v)\leq LN_x(v)$$
    for some constant $C(L)>0$ with $C(1)=1$. This implies that there exists a curve family $\Gamma_3$ in $U$ with $\mod \Gamma_3=0$ such that for all curves $\beta\colon [a,b]\to U$ parametrized by arclength that are outside $\Gamma_3$ we have
    \begin{align}\label{length:ineq}
        C(L)\int_a^b{N}_{\beta(t)}(\dot\beta(t))\,dt \leq \int_a^b\widetilde{N}_{\beta(t)}(\dot\beta(t))\,dt\leq L \int_a^b{N}_{\beta(t)}(\dot\beta(t))\,dt.
    \end{align}
    Let $\Gamma_0$ be the family of curves that have a reparametrization in $\Gamma_1\cup \Gamma_2\cup \Gamma_3$. Then $\mod \Gamma_0=0$. By combining \eqref{length:gamma}, \eqref{length:fgamma}, and \eqref{length:ineq}, we see that the conclusions of the lemma are true for the family $\Gamma_0$. 
    
    Now, we prove the claim. Theorem \ref{thm:area-formula} applied to $\phi$ provides a set of measure zero $G_1\subset U$ such that for any measurable set $A\subset U\setminus G_1$ we have 
    \begin{align}\label{lemma:length:jx}
       \int_AJ_x&= \int_X N(\phi|_A,x)\, d\mathcal H^2=\hm(\phi(A)),
    \end{align}
    where the last equality follows from Theorem \ref{theorem:radon-nikodym}.
    Similarly,
    there exists a set $G_2\subset U$ of measure zero such that for any measurable set $A\subset U\setminus G_2$,
    \begin{align*}
        \int_A\widetilde{J}_x &= \int_Y \widetilde N(f\circ \phi|_A ,y)\, d\mathcal H^2.
    \end{align*}
    From Lemma \ref{lemma:multiplicity} we know that $N(f,y)=1$ for a.e.\ $y\in f(X)$. By Theorem \ref{theorem:radon-nikodym}, for a.e.\ $x\in X$, $\phi^{-1}(x)$ is a singleton. Since $f$ is area-preserving and in particular has the Lusin (N) property, we conclude that for a.e.\ $y\in f(X)$ the set $\phi^{-1}(x)$ is a singleton whenever $f(x)=y$. In summary, $N(f\circ\phi|_A,y)=1$ for a.e.\ $y\in f(\phi(A))$. In particular, for any measurable set $A\subset U\setminus G_2$,
    \begin{align*}
        \int_A\widetilde{J}_x = \int_Y \widetilde N(f\circ \phi|_A ,y)\, d\mathcal H^2= \hm(f(\phi(A))).
    \end{align*}
    The area-preserving property of $f$ and \eqref{lemma:length:jx} now imply that $J_x=\widetilde{J}_x$ for a.e.\ $x\in U$ and hence
    \begin{align}\label{eq:ball}
        |B_x|=|\widetilde{B}_x|
    \end{align}
    for a.e.\ $x\in U$. This  equality implies that $N_x$ is not a norm if and only if $\widetilde N_x$ is also not a norm. By Lemma \ref{lemma:wqc_derivative}, if $N_x$ is not a norm, then $N_x\equiv 0$. 
    
    Let $K_i,\widetilde K_j\subset U$, $i,j\in\N$, be the sets from Proposition \ref{prop:decomp-apmd} applied to $\phi,f\circ\phi$, respectively. Let $\varepsilon>0$. The Lipschitz property of $f$ implies that
    $$\widetilde{N}_x(v)\leq L N_x(v)+(1+L)\varepsilon|v|$$
    for every $x\in K_{i,j}=K_i\cap\widetilde{K}_j$ and every $v\in\R^2$ with $|v|\leq\min\{r_i(\varepsilon),\widetilde{r}_j(\varepsilon)\}$ and $x+v\in K_{i,j}$. This shows that 
    \begin{align}\label{ineq:Lapmd}
        \widetilde{N}_x(v)\leq L N_x(v) \quad \textrm{and thus} \quad B_x\subset L\widetilde B_x
    \end{align}
    for a.e.\ $x\in U$ and all $v\in\R^2$. Here $L\widetilde B_x$ denotes the closed ball $\{y\in \R^2: \widetilde N_x(y)\leq L\}$. In particular, if $N_x$ is not a norm, then $\widetilde N_x\equiv N_x\equiv 0$.
    
    If $L=1$, then \eqref{ineq:Lapmd} implies that $B_x\subset \widetilde B_x$ for a.e.\ $x\in U$. By \eqref{eq:ball}, we have $B_x=\widetilde B_x$ for a.e.\ $x\in U$, since $N_x$ and $\widetilde N_x$ are either both norms or vanish identically. Hence, $N_x(v)=\widetilde{N}_x(v)$ for a.e.\ $x\in U$ and all $v\in\R^2$. 

    Denote by $E_x,\widetilde{E}_x$ the John ellipse of $B_x,\widetilde{B}_x$, respectively, whenever $N_x$ and $\widetilde N_x$ are norms. John's Theorem (see \cite{Bal97}*{Theorem 3.1}) implies that 
    \begin{align}\label{eq:ellipse}
        E_x\subset B_x\subset\sqrt{2}E_x\quad\text{and}\quad\widetilde{E}_x\subset\widetilde{B}_x\subset\sqrt{2}\widetilde{E}_x.
    \end{align}
    Denote by $a_x,\widetilde{a}_x$ (resp.\ $b_x,\widetilde b_x$) the length of the major (resp.\ minor) axis of $E_x,\widetilde{E}_x$, respectively. By \eqref{ineq:Lapmd} and \eqref{eq:ellipse} we have that
    $$L^{-1}E_x\subset L^{-1}B_x\subset\widetilde{B}_x\subset\sqrt{2}\widetilde{E}_x,$$
    which implies that $b_x\leq\sqrt{2}L\widetilde{b}_x$. Moreover, combining \eqref{eq:ball} and \eqref{eq:ellipse} gives
    $$|\widetilde{E}_x|\leq|\widetilde{B}_x|=|B_x|\leq2|E_x|.$$ 
    Since $|E_x|=\pi a_x b_x$ and $|\widetilde{E}_x|=\pi\widetilde{a}_x\widetilde{b}_x$, we get
    $$\widetilde{a}_x\leq2\,\frac{a_x b_x}{\widetilde{b}_x}\leq2\sqrt{2}La_x.$$
    In particular, if we assume in addition that $E_x$ is a geometric ball, then $\widetilde{E}_x\subset 2\sqrt{2}LE_x$. All in all we obtain that 
    \begin{align}\label{lemma:length:conclusion}
            L^{-1}B_x\subset\widetilde{B}_x\subset\sqrt{2}\widetilde{E}_x\subset4LE_x\subset4LB_x,
    \end{align}
    with the additional assumption that $E_x$ is a geometric ball. Note that \eqref{lemma:length:conclusion} shows that the claim holds for $C(L)=(4L)^{-1}$. 

    For the general case that $E_x$ is not a geometric ball, we consider a linear map $T\colon \R^2\to \R^2$ such that $T(E_x)$ is a round ball. Note that \eqref{eq:ball} remains true for the images of $B_x,\widetilde{B}_x$ under $T$. Since the John ellipse is preserved under linear maps, the above calculations are true for the images of the corresponding sets under $T$, and thus one obtains the inclusions \eqref{lemma:length:conclusion} for the images. Therefore, the inclusions also hold for the original sets.
    \end{proof}

\begin{proof}[Proof of Theorem \ref{thm:main} \ref{main:length}]
    We cover $X$ with a countable collection of open sets $\{X_n\}_{n\in \N}$, each homeomorphic to $\R^2$. Every $X_n$ is reciprocal and, by Rajala's theorem \cite{Raj:17}, there exists a quasiconformal homeomorphism $\phi_n\colon U_n\to X_n$, where $U_n\subset \R^2$ is an open set. By Lemma \ref{lemma:length}, 
    \begin{align*}
        C(L)\ell(\phi_n\circ \beta)\leq \ell(f\circ \phi_n\circ \beta)\leq L\ell(\phi_n\circ \beta)
    \end{align*}
    holds for every curve $\beta$ in $U_n$ outside a curve family $\Gamma_n$ with $\mod\Gamma_n=0$, where $C(L)>0$ is some constant with $C(1)=1$. Since $\phi_n$ is quasiconformal, $\mod\phi_n(\Gamma_n)=0$ for each $n\in \N$. Note that if $\gamma$ is a curve in $X_n$ outside $\phi_n(\Gamma_n)$, then after setting $\beta=\phi_n^{-1}\circ\gamma$ we see that the statement of Theorem \ref{thm:main} \ref{main:length} holds for $\gamma$. We define $\Gamma_0$ to be the family of curves in $X$ that have a subcurve in some $\phi_n(\Gamma_n)$, $n\in \N$. Then $\mod \Gamma_0=0$ and the conclusions of Theorem \ref{thm:main} \ref{main:length} hold for all curves $\gamma$ in $X$ outside $\Gamma_0$. 
\end{proof}

\subsection{Injectivity}
In this section we establish Theorem \ref{thm:main} \ref{main:injective}. The main difficulty is to establish the injectivity of $f$. A map $f\colon X\to Y$ is \textit{light} if $f^{-1}(y)$ is totally disconnected for each $y\in Y$.

\begin{lemma}\label{lemma:light}
    Suppose that $Y$ is reciprocal. Let $f\colon X\to Y$ be a non-constant continuous map such that there exists $K>0$ with the property that $\mod \Gamma\leq K\mod f(\Gamma)$ for each curve family $\Gamma$ in $X$. Then $f$ is a light map. 
\end{lemma}

\begin{proof}
Let $y\in Y$ and suppose that $f^{-1}(y)$ contains a non-degenerate continuum $E$. Consider a non-degenerate continuum $F\subset X\setminus f^{-1}(y)$; note that the latter set is non-empty because $f$ is non-constant. The family $\Gamma$ of curves joining $E$ and $F$ has positive modulus \cite{Raj:17}*{Proposition 3.5}. On the other hand, each curve of $f(\Gamma)$ joins the continuum $f(F)$ to $y$. Since $Y$ is reciprocal, we have $\mod f(\Gamma)=0$ (see \eqref{reciprocality:3}). This is a contradiction.  
\end{proof}

For $y_0\in Y$ and $r>0$ we denote by $S(y_0,r)$ the set $\{y\in Y: d(y,y_0)=r\}$. 

\begin{lemma}\label{lemma:circle_jordan}
    Let $y_0\in Y$ and $K\subset Y\setminus \{y_0\}$ be a closed set. There exists $\delta>0$ such that for a.e.\ $r\in (0,\delta)$ there exists a component $E\subset S(y_0,r)$ that is a rectifiable Jordan curve separating $y_0$ and $K$. 
\end{lemma}
\begin{proof}
    Let $U\subset Y$ be the interior of a topological closed disk $\bar U\subset Y$ such that $Y\setminus U$ is connected, $y_0\in U$, and $K\subset Y\setminus U$. Note that $\mathcal H^2(\bar U)<\infty$. Let $\delta>0$ be sufficiently small such that $\bar B(y_0,\delta)\subset U$. Then for all $r\in (0,\delta)$ the set $S(y_0,r)$ is compact. By the coarea inequality for Lipschitz functions (Theorem \ref{theorem:coarea:classical}), $\mathcal H^1(S(y_0,r))<\infty$ for a.e.\ $r\in (0,\delta)$. By Corollary \ref{corollary:parametrizations} (see also \cite{Nt:monotone}*{Theorem 1.5}), for a.e.\ $r\in (0,\delta)$, each component of $S(y_0,r)$ is a rectifiable Jordan arc or Jordan curve. Fix such a parameter $r$. Since $S(y_0,r)$ separates $y_0$ from all points of $\partial U$, by Lemma \ref{lemma:separation_component} there exists a component $E$ of $S(y_0,r)$ that separates $y_0$ from $\partial U$. In particular, $E$ must be a Jordan curve and separates $y_0$ from $K$.
\end{proof}

\begin{lemma}\label{lemma:preimage_jordan}
Let $Z\subset X$ be homeomorphic to a topological closed disk and let $f\colon Z\to Y$ be a continuous map in $N^{1,2}(Z,Y)$. For every $y_0\in Y$ and for a.e.\ $r\in (0,\infty)$, each component of $f^{-1}(S(y_0,r))$ is a Jordan arc or a Jordan curve.
\end{lemma}
\begin{proof}
    Define $u(x)=d(f(x),y_0)$ on $Z$, which is continuous and lies in $N^{1,2}(Z)$. Observe that $u^{-1}(r)=f^{-1}(S(y_0,r))$ for every $r>0$. By the coarea inequality of Theorem \ref{theorem:coarea:introduction} we see that  $\mathcal H^1(u^{-1}(r)\cap \mathcal A_u)<\infty$ for a.e.\ $r>0$. In particular, for such values $r$, if $E$ is a non-degenerate component of $u^{-1}(r)$, then $E\subset \mathcal A_u$, so $\mathcal H^1(E)<\infty$. By Corollary \ref{corollary:parametrizations}, for a.e.\ $r>0$, every non-degenerate component of $u^{-1}(r)$ is a Jordan arc or a Jordan curve. 
\end{proof}

\begin{lemma}\label{lemma:injectivity}
    Let $f\colon X\to Y$ be a continuous light map in $N^{1,2}(X,Y)$ such that $N(f,y)\leq 1$ for a.e.\ $y\in Y$. Then $N(f,y)\leq 1$ for every $y\in Y$. In particular, $f$ is injective. 
\end{lemma}

\begin{proof}
Let $y\in f(X)$ and $x\in f^{-1}(y)$. For the moment, we consider the restriction $g=f|_Z$ to a compact neighborhood $Z$ of $x$ that is homeomorphic to a closed disk and contains $x$ in its interior. Since $g$ is light, it is non-constant on $\inter(Z)$ and there exists a point $z\in \inter(Z)\setminus  g^{-1}(y)$. Note that for each $r\in (0, d(y,g(z)))$ the set $S(y,r)$ separates $y$ from $g(z)$. Therefore, the compact set $g^{-1}(S(y,r))$ separates $x$ from $z$. By Lemma \ref{lemma:preimage_jordan}, for a.e.\ $r>0$, each component of $g^{-1}(S(y,r))$ is a Jordan arc or a Jordan curve. Combining these facts with Lemma \ref{lemma:separation_component}, we see that there exists a full measure subset $I$ of $(0,d(y,g(z)))$ such that for each $r\in I$, there exists a component of $g^{-1}(S(y,r))$ that separates $x$ from $z$ and is a Jordan arc or a Jordan curve. 

We claim for all sufficiently small $r\in I$, each such component must be a Jordan curve. To prove this, suppose that there exists a sequence of positive numbers $r_n\to 0$ and components $F_{r_n}$  of $g^{-1}(S(y,r_n))$ that are Jordan arcs and separate $x$ from $z$. Fix a continuum $K\subset \inter(Z)$ connecting $x$ and $z$. Since $F_{r_n}$ separates $x$ from $z$, it intersects $K$. Moreover, since $F_{r_n}$ is a Jordan arc, it cannot be contained in $\inter(Z)$, as $\inter(Z)\setminus F_{r_n}$ would then be connected. Therefore, $F_{r_n}$ intersects $\partial Z$ and 
$$\diam(F_{r_n})\geq \dist(K,\partial Z)>0$$ 
for all $n\in \N$. After passing to a subsequence, $F_{r_n}$ converges in the Hausdorff sense to a non-degenerate continuum $F$. Since $r_n\to 0$, we have that $F\subset g^{-1}(y)$. This contradicts the lightness of $g$. The claim is proved. 

By the assumption that $N(f,w)\leq 1$ for a.e.\ $w\in Y$ and the coarea inequality for Lipschitz functions (Theorem \ref{theorem:coarea:classical}), we see that for a.e.\ $r>0$, $\mathcal H^1$-a.e.\ point of $S(y,r)$ has at most one preimage under $f$. Also, given a closed set $K\subset Y\setminus \{y\}$, by Lemma \ref{lemma:circle_jordan}, for a.e.\ sufficiently small $r>0$ there exists a Jordan curve $E\subset S(y,r)$ separating $y$ from $K$.  Altogether, there exists $\delta'>0$ and a set $I'\subset (0,\delta')$ of full measure so that for every $r\in I'$ the following statements are true. 
\begin{enumerate}
    \item $\mathcal H^1$-a.e.\ point of $S(y,r)$ has at most one preimage under $f$.\label{inj:i}
    \item There exists a component of $S(y,r)$ that is a Jordan curve separating $y$ and $K$.\label{inj:ii}
    \item Each component of $g^{-1}(S(y,r))$ that separates $x$ and $z$ is a Jordan curve. \label{inj:iii}
\end{enumerate}

Let $E$ be a component of $S(y,r)$, $r\in I'$, that is a Jordan curve and let $F\subset g^{-1}(E)$ be a Jordan curve. We claim that $g(F)=E$. By \eqref{inj:i}, $\mathcal H^1$-a.e.\ point of $E$ has at most one preimage under $g$. The map $g|_F$ is conjugate to a continuous map $\phi \colon \mathbb S^1\to \mathbb S^1$ with the property that a dense set of points of $\mathbb S^1$ have at most one preimage. Suppose that $g(F)$ is a strict subarc of $E$. Note that $g(F)$ cannot be a point since $g$ is light. Then  $\phi(\mathbb S^1)$ is a non-degenerate strict subarc of $\mathbb S^1$. This contradicts the fact that a dense set of points of $\mathbb S^1$ have at most one preimage. We have shown the following.
\begin{enumerate}\setcounter{enumi}{3}
    \item If $E$ is a component of $S(y,r)$ that is a Jordan curve and $F\subset g^{-1}(E)$ is a Jordan curve, then $g(F)=E$. \label{inj:iv}
\end{enumerate}

We have completed our preparation to show the injectivity of $f$.
Suppose that $f^{-1}(y)$ contains two points $x_1,x_2$ for some $y\in f(X)$. We consider disjoint topological closed disks $Z_1,Z_2\subset X$ such that $x_i\in \inter(Z_i)$, $i=1,2$. We also fix $z_i\in \inter(Z_i)\setminus f^{-1}(y)$. Consider the restrictions $g_i=f|_{Z_i}$, $i=1,2$. By the previous, for $i=1,2$, there exists a set $I_i'$ of full measure in an interval $(0,\delta_i')$, such that \eqref{inj:i}--\eqref{inj:iv} are true for the map $g_i$; specifically, in \eqref{inj:ii} we use the set  $K=\{f(z_1),f(z_2)\}$. Let $I'=I_1'\cap I_2'$, which has full measure in $(0,\delta')$, where $\delta'=\min \{\delta_1',\delta_2'\}$. By \eqref{inj:ii}, for $r\in I'$ there exists a component $E$ of $S(y,r)$ that is a Jordan curve separating each of the pairs $(y,f(z_1))$ and $(y,f(z_2))$. Let $F_i$ be a component of $g_i^{-1}(E)$ that separates $x_i$ and $z_i$, $i=1,2$; such components exist by Lemma \ref{lemma:separation_component}. Note that $F_i$ is also a component of $g_i^{-1}(S(y,r))$. By \eqref{inj:iii}, $F_i$ is a Jordan curve for $i=1,2$. By \eqref{inj:iv}, we conclude that $g_i(F_i)=E$, $i=1,2$. Thus, each point of $E$ has at least two preimages under $f$. This contradicts \eqref{inj:i}.
\end{proof}

\begin{lemma}\label{lemma:qc_lengths}
   Let $f\colon X\to Y$ be an area-preserving map that is a quasi\-conformal homeomorphism. Suppose that there exists $K\geq 1$ such that $${K}^{-1/2}\ell(\gamma)\leq \ell(f\circ \gamma)\leq {K}^{1/2} \ell (\gamma)$$
for all curves $\gamma$ in $X$ outside a curve family $\Gamma_0$ with $\mod\Gamma_0=0$.  Then $f$ is $K$-quasiconformal. 
\end{lemma}
\begin{proof}
    The constant function $K^{1/2}$ is a $2$-weak upper gradient of $f$ in $N^{1,2}_{\loc}(X)$. Moreover, by the preservation of area, for each Borel set $E\subset Y$ we have
    \begin{align*}
        \int_{f^{-1}(E)} K \, d\mathcal H^2= K \mathcal H^2(f^{-1}(E))=K\mathcal H^2(E).
    \end{align*}
    In view of Theorem \ref{theorem:definitions_qc}, we derive that $f$ is weakly $K$-quasiconformal. Since $f$ is quasiconformal, we have
    $$\ell(f^{-1}\circ \gamma)\leq {K}^{1/2} \ell (\gamma) $$
    for all curves $\gamma$ in $Y$ outside a curve family $\Gamma_0'$ with $\mod\Gamma_0'=0$. Thus, the same argument applies to $f^{-1}$ and shows that it is weakly $K$-quasiconformal. Altogether, $f$ is $K$-quasiconformal.  
\end{proof}

\begin{proof}[Proof of Theorem \ref{thm:main} \ref{main:injective}]
    Suppose that $f$ is $L$-Lipschitz and area-pre\-serv\-ing. By Lemma \ref{lemma:multiplicity}, $N(f,y)=1$ for a.e.\ $y\in f(X)$. Also, Lemmas \ref{lemma:weak_qc} and \ref{lemma:light} imply that $f$ is a light map. Now, Lemma \ref{lemma:injectivity} implies that the restriction of $f$ to any precompact open subset $U$ of $X$ (so that $f|_U\in N^{1,2}(U,Y)$) is injective. This implies that $f$ is injective in all of $X$. The invariance of domain theorem implies that $f$ is an embedding. Since $f$ is surjective by assumption, we conclude that $f$ is a homeomorphism.  By Lemma \ref{lemma:weak_qc}, we see that $f$ is a weakly $L^2$-quasiconformal homeomorphism. Since $Y$ is reciprocal, Lemma \ref{lemma:weak_qc_upgrade} yields that $f$ is $K$-quasiconformal for some $K=K(L)\geq 1$. In particular, this implies that $X$ is also reciprocal. 
    
    The final inequality in Theorem \ref{thm:main} \ref{main:injective} involving the lengths follows from Theorem \ref{thm:main} \ref{main:length}. In the case that $f$ is $1$-Lipschitz, we obtain $\ell(\gamma)=\ell(f\circ \gamma)$ for all curves $\gamma$ in $X$ outside a curve family $\Gamma_0$ with $\mod\Gamma_0=0$. By Lemma \ref{lemma:qc_lengths}, we conclude that $f$ is $1$-quasiconformal.    
\end{proof}

\subsection{Bounded length distortion and isometry}
Here we prove Theorem \ref{thm:main} \ref{main:bld}. Our goal is to upgrade the conclusion of Theorem \ref{thm:main} \ref{main:injective} so that the length of every path, rather than almost every path, is quasi-preserved. This is achieved with the aid of upper Ahlfors $2$-regularity. We say that a space is locally upper Ahlfors $2$-regular with constant $K>0$ if each point has a neighborhood $U$ such that $\mathcal H^2(B(x,r))\leq Kr^2$ for all $x\in U$ and $r<\diam(U)$. We denote by $N_r(E)$ the open $r$-neighborhood of a set $E$.

\begin{lemma}\label{lemma:upper_mass}
    Suppose that $Y$ is locally upper Ahlfors $2$-regular with constant $K>0$ and $\gamma$ is a curve in $Y$. Then for all sufficiently small $r>0$ we have
    $$\hm(N_r(|\gamma|))\leq 2Kr\ell(\gamma)+ 8Kr^2.$$
\end{lemma}

\begin{proof}
Without loss of generality, $\gamma\colon [0,\ell(\gamma)]\to X$ is non-constant, rectifiable and para\-metrized by arclength. Assume that $0<r<\ell(\gamma)/2$ and that  for every $x\in|\gamma|$ we have
    $$\hm(B(x,2r))\leq4Kr^2.$$ 
Consider a partition $\{t_0,\dots,t_n\}$ of $[0,\ell(\gamma)]$ such that $|t_{i}-t_{i-1}|\leq 2r$, $i\in \{1,\dots,n\}$, and $(n-1)2r< \ell(\gamma)\leq 2nr$.  Then $\{B( \gamma(t_i),2r)\}_{i=0}^n$ covers $N_r(|\gamma|)$ and we can compute
    \[\hm(N_r(|\gamma|))\leq\sum_{i=0}^n\hm(B(\gamma(t_i),2r))\leq (n+1)4Kr^2\leq 2Kr\ell(\gamma)+8Kr^2.\qedhere\]
\end{proof}

\begin{lemma}\label{lemma:nearby_curves}
    Suppose that $Y$ is locally upper Ahlfors $2$-regular with constant $K>0$. Let $\Gamma_0$ be a curve family in $Y$ with $\mod\Gamma_0=0$. Then for each curve $\gamma\colon [a,b]\to Y$ and for each $\varepsilon>0$ there exists a curve $\gamma_{\varepsilon}\colon [a,b]\to Y$ with the following properties.
    \begin{enumerate}[label=\normalfont(\arabic*)]
        \item $\gamma_{\varepsilon}\notin \Gamma_0$.\label{nearby_curves:Gamma0}
        \item $|\gamma(a)-\gamma_{\varepsilon}(a)|<\varepsilon$, $|\gamma(b)-\gamma_{\varepsilon}(b)|<\varepsilon$, and $|\gamma_{\varepsilon}|\subset N_{\varepsilon}(|\gamma|)$. \label{nearby_curves:nbhd}
        \item $\ell(\gamma_{\varepsilon}) \leq 4\pi^{-1} K\ell(\gamma)+\varepsilon$. \label{nearby_curves:length}
    \end{enumerate}
    Moreover, if $Y$ is Riemannian, then 
    \begin{enumerate}
        \item[\normalfont{(3')}] $\ell(\gamma_{\varepsilon}) \leq  \ell(\gamma)+\varepsilon$. \label{nearbycurves:Riem} 
    \end{enumerate} 
\end{lemma}
\begin{proof}
    Assume that $\gamma$ is simple, otherwise we consider a simple curve with trace in $|\gamma|$ connecting $\gamma(a)$ and $\gamma(b)$. Let $\varepsilon>0$.  Consider the distance function $g(x)=d(x,|\gamma|)$. By the coarea inequality for Lipschitz functions (Theorem \ref{theorem:coarea:classical}) and Lemma \ref{lemma:upper_mass}, there exists $r_1>0$ such that for all $0<r<r_1$ we have
\begin{align}\label{ineq:nearby_curves}
    \int\displaylimits^* \chi_{(0,r)}(t)\mathcal H^1(g^{-1}(t))\, dt \leq\frac{4}{\pi} \mathcal H^2(N_r(|\gamma|)) <\frac{8}{\pi}Kr \ell(\gamma)+\varepsilon r.
\end{align}
Therefore, for all $0<r<r_1$ we have
\begin{align}
    \essinf_{t\in (0,r)} \mathcal H^1(g^{-1}(t))<\frac{8}{\pi}K\ell(\gamma)+\varepsilon. 
\end{align}
By Lemma \ref{lemma:avoid}, for a.e.\ $t\in (0,r_1)$, every Lipschitz and injective curve $\alpha\colon [a,b]\to g^{-1}(t)$ does not lie in $\Gamma_0$.

    Let $U\subset Y$ be a neighborhood of $|\gamma|$ homeomorphic to $\D$. Since $\gamma$ is simple, the space $Z:=U/|\gamma|$ equipped with the quotient metric is homeomorphic to $\D$. The quotient map $\pi\colon U\to Z$ is a local isometry on $U\setminus|\gamma|$. This together with Lemma \ref{lemma:circle_jordan} provides the existence of $r_2>0$ such that for a.e.\ $t\in(0,r_2)$, the level set $g^{-1}(t)$ contains a rectifiable Jordan curve $\gamma_t$ in $U$ separating $|\gamma|$ from $\partial U$. Note that $|\gamma_t|$ converges to $|\gamma|$ in the Hausdorff sense as $t\to 0$. Thus, there exists $r_3\in (0,r_2)$ such that for a.e.\ $t\in(0,r_3)$ we can find distinct points $a_t\in\overline{B}(\gamma(a),\varepsilon)\cap |\gamma_t|$ and $b_t\in\overline{B}(\gamma(b),\varepsilon)\cap |\gamma_t|$. Let $\gamma_t'$ be a Lipschitz and injective parametrization of the closure of the shorter component of $|\gamma_t|\setminus\{a_t,b_t\}$. For $0<r<\min \{r_1,r_2,r_3,\varepsilon\}$ we have
    \begin{align*}
    \essinf_{t\in (0,r)} \ell(\gamma_t') <\frac{4}{\pi}K\ell(\gamma)+\frac{\varepsilon}{2}. 
\end{align*}
    By the previous, $\gamma_t'\notin \Gamma_0$ for a.e.\ $t\in (0,r)$. Moreover, $|\gamma_t'|\subset g^{-1}(t)\subset N_{\varepsilon}(|\gamma|)$. Therefore, there exists $t\in (0,r)$ so that $\gamma_t'$ satisfies \ref{nearby_curves:Gamma0}--\ref{nearby_curves:length}.

If $Y$ is Riemannian we have a local upper area bound of the form 
$$\hm(N_r(|\gamma|))\leq2r\ell(\gamma)+O(r^2);$$
see \cite{Gra04}*{Corollary 9.24}.
By arguing as in \eqref{ineq:nearby_curves} while applying the coarea inequality for Riemannian manifolds (Theorem \ref{theorem:coarea:classical}), we obtain
\begin{align*}
    \essinf_{t\in (0,r)} \mathcal \ell(\gamma_t') \leq \ell(\gamma)+\varepsilon,
\end{align*}
for all sufficiently small $r>0$. Hence, \hyperref[nearbycurves:Riem]{(3')} follows.
\end{proof}

\begin{lemma}\label{lemma:BLD}
    Suppose that $Y$ is locally upper Ahlfors 2-regular with constant $K>0$. Let $g\colon Y\to X$ be continuous map such that there exists $L>0$ with the property that $\ell(g\circ \gamma)\leq L \ell(\gamma)$ for all curves $\gamma$ in $Y$ outside a curve family $\Gamma_0$ with $\mod\Gamma_0=0$. Then
    $$\ell(g\circ \gamma)\leq\frac{4}{\pi}KL\ell(\gamma)$$
    for every rectifiable curve $\gamma$ in $Y$. Moreover, if $Y$ is Riemannian then 
    $$\ell(g\circ \gamma)\leq L\ell(\gamma)$$
    for every rectifiable curve $\gamma$ in $Y$.
\end{lemma}
\begin{proof}
    Let $\gamma$ be a rectifiable Jordan arc in $Y$. By Lemma \ref{lemma:nearby_curves}, for each $n\in \N$ we can find a curve $\gamma_{n}\subset N_{1/n}(|\gamma|)$ whose endpoints are $(1/n)$-close to the endpoints of $\gamma$, $\gamma_n\notin \Gamma_0$, and
    $$\ell(\gamma_{n}) \leq {4}{\pi^{-1}}K\ell(\gamma)+n^{-1}.$$ 
    Suppose that $\gamma_n$ is parametrized by $[0,1]$ with constant speed. After passing to a subsequence, we may assume that $\gamma_n$ converges uniformly to a path $\widetilde \gamma\colon [0,1]\to |\gamma|$ with the same endpoints as $\gamma$. It follows that $\widetilde \gamma$ is surjective, but it is possibly not injective. Moreover, $g\circ \gamma_n$ converges uniformly to $g\circ \widetilde \gamma$.  Since $\gamma$ is a Jordan arc, we have $N(g\circ \widetilde \gamma ,y)\geq N(g\circ \gamma,y)$ for each $y\in g(|\gamma|)$. The area formula for length \cite{Fed69}*{Theorem 2.10.13} and the lower semi-continuity of length imply that
    \begin{align*}
        \ell(g\circ \gamma)\leq \ell(g\circ \widetilde \gamma)\leq \liminf_{n\to\infty}\ell(g\circ \gamma_n).
    \end{align*}
    Since $\gamma_n\notin \Gamma_0$, the latter is bounded by 
    \begin{align*}
        L\liminf_{n\to\infty} \ell(\gamma_n) \leq {4}{\pi}^{-1}KL\ell(\gamma).
    \end{align*}
    This completes the proof in the case of Jordan arcs. 
     
    Now, suppose that $\gamma\colon [a,b]\to Y$ is an arbitrary path. Let $\{t_0,\dots,t_n\}$ be a partition of $[a,b]$. For $i\in \{1,\dots,n\}$, let $\gamma_i\colon [t_{i-1},t_i] \to  \gamma([t_{i-1},t_i])$ be a Jordan arc with endpoints $\gamma(t_{i-1}),\gamma(t_i)$. Then 
    \begin{align*}
        \sum_{i=1}^n d(g(\gamma(t_{i-1})),g(\gamma(t_i)))&\leq \sum_{i=1}^n \ell(g\circ \gamma_i) \leq {4}{\pi}^{-1}KL \sum_{i=1}^n  \ell(\gamma_i)\\
        &\leq {4}{\pi}^{-1}KL\sum_{i=1}^n \ell(\gamma|_{[t_{i-1},t_i]}) = {4}{\pi}^{-1}KL\ell(\gamma).
    \end{align*}
    This yields $\ell(g\circ\gamma)\leq 4\pi^{-1}KL\ell(\gamma)$.

    If $Y$ is Riemannian, the statement follows after applying \hyperref[nearbycurves:Riem]{(3')} from Lemma \ref{lemma:nearby_curves} instead of \ref{nearby_curves:length}.
\end{proof}

\begin{proof}[Proof of Theorem \ref{thm:main} \ref{main:bld}]
The upper Ahlfors $2$-regularity implies that $Y$ is reciprocal \cite{Raj:17}*{Theorem 1.6}. By Theorem \ref{thm:main} \ref{main:injective}, we have that $f$ is a quasiconformal homeomorphism and the length of a.e.\ path is quasi-preserved. We now apply Lemma \ref{lemma:BLD} to $g=f^{-1}$, together with the fact that $f$ is Lipschitz, and conclude that the length of every rectifiable path is quasi-preserved. It also follows that $\ell(\gamma)<\infty$ if and only if $\ell(f\circ \gamma)<\infty$. Therefore, $f$ is a map of bounded length distortion. 
\end{proof}

\begin{proof}[Proof of Theorem \ref{thm:main} \ref{main:isometry}]
    Since $Y$ is reciprocal, by Theorem \ref{thm:main} \ref{main:injective}, $f$ is a $1$-quasiconformal homeomorphism and preserves the length of all curves in $X$ outside a curve family $\Gamma_0$ with $\mod\Gamma_0=0$.  It follows from Lemma \ref{lemma:BLD} that $\ell(f^{-1}\circ \gamma)\leq\ell(\gamma)$  for every rectifiable curve $\gamma$ in $Y$. If $x,y\in X$ and $\gamma$ is a rectifiable curve in $Y$ joining $f(x)$ and $f(y)$ then
   $$d(x,y)\leq\ell(f^{-1}\circ\gamma)\leq\ell(\gamma).$$
   Infimizing over $\gamma$ gives $d(x,y)\leq d(f(x),f(y))$. Equality follows from $f$ being 1-Lipschitz.
\end{proof}

\section{Examples}\label{section:examples}
We present examples that show the optimality of Theorem \ref{thm:main}. In all examples $X,Y$ are metric surfaces of locally finite Hausdorff $2$-measure and $f\colon X\to Y$ is an area-preserving and $1$-Lipschitz map. 

\begin{example}\label{example:collapse}
    This example shows that $f$ is not a homeomorphism in general, even if $X$ is Euclidean. Let $I$ be the interval $[0,1]\times \{0\}$ and $Y=\R^2/I$, equipped with the quotient metric. The natural projection map $f\colon\R^2\to Y$ is area-preserving and $1$-Lipschitz, but it is not a homeomorphism. 
\end{example}

\begin{example}\label{example:weight}
    This example shows that if $Y$ is reciprocal as in Theorem \ref{thm:main} \ref{main:injective}, then $f$ is not BLD in general, even if $X$ is Euclidean. Define the weight $\omega\colon\R^2\to[0,1]$ by $\omega(x)= x_1$ if $x=(x_1,0)\in I\coloneqq(0,1]\times \{0\}$ and $\omega(x)=1$ otherwise. We define a metric $d$ on $\R^2$ by
    $$d(x,y):=\inf_{\gamma} \int_{\gamma}\omega\, ds,$$
    where the infimum is taken over all rectifiable curves $\gamma$ connecting $x,y\in\R^2$. Let $f\colon\R^2\to Y:=(\R^2,d)$ be the identity map, which is 1-Lipschitz, since $\omega\leq 1$, and a local isometry on $\R^2\setminus I$, hence area-preserving. Moreover, $f$ is a homeomorphism, and thus $Y$ is a metric space homeomorphic to $\R^2$. 
    
    One can show that for each Borel set $E\subset \R^2$ we have $\mathcal H^1_d(E)= \int_E \omega \, d\mathcal H^1$; in fact, it suffices to show this for sets $E\subset I$.  This fact and the area formula for length \cite{Fed69}*{Theorem 2.10.13} imply that if $\gamma$ is a rectifiable curve with respect to the Euclidean metric, then $\ell_{d}(\gamma)=\int_{\gamma}\omega\, ds$. This implies that 
    $$ \int_{\gamma} \rho \, ds_d = \int_{\gamma} \rho \omega\, ds$$
    for every Borel function $\rho\colon \R^2\to [0,\infty]$. 
    
    Let $\Gamma$ be a family of curves in $\R^2$. Since $f$ is $1$-Lipschitz and area-preserving, by Lemma \ref{lemma:weak_qc} we have $\mod \Gamma\leq\mod f(\Gamma)$; here the latter modulus is with respect to the metric $d$. We now show the reverse inequality. Let $\rho\colon \R^2\to [0,\infty]$ be admissible for $\Gamma$. We set $\rho'=\rho \omega^{-1}$. If $\gamma\in \Gamma$, then 
    $$\int_{f\circ \gamma} \rho' \, ds_{d} = \int_{\gamma} \rho \omega^{-1}\omega \, ds =\int_{\gamma}\rho\, ds \geq 1.$$
    Thus, $\rho'$ is admissible for $f(\Gamma)$. Since $\mathcal H^2_d(I)=0$, we conclude that $$\mod f(\Gamma)\leq \int \rho^2\, d\mathcal H^2$$
    and thus $\mod f(\Gamma)\leq \mod \Gamma$.  This shows that $f$ is $1$-quasiconformal and that $Y$ is reciprocal. 
    
    By Theorem \ref{thm:main} \ref{main:length}, $f$ preserves the length of a.e.\ curve with respect to $2$-modulus; this can also be seen immediately here, since a.e.\ curve intersects $I$ at a set of length zero.  However, $f$ does not preserve the length of \textit{every} curve and is not BLD. Indeed, for $t\in(0,1]$ denote by $\gamma_t$ the straight line segment connecting $(0,0)$ and $(t,0)$. Then $\ell(\gamma_t)=t$, whereas
    $$\ell_{d}(\gamma_t)=\int_{\gamma_t}\omega\, ds=t^2/2.$$
\end{example}

\section{Coarea inequality}
In this section we establish the general coarea inequality of Theorem \ref{theorem:coarea:introduction}. First we prove the statement in the case that $X$ is a topological closed disk. The proof follows the same strategy as in \cite{EIR22}*{Theorem 4.8}.

\begin{thm}\label{theorem:coarea}
    Let $X$ be a metric surface of finite Hausdorff $2$-measure that is homeomorphic to a topological closed disk and suppose that there exists a weakly $K$-quasiconformal map from $\bar \D$ onto $X$ for some $K\geq 1$. Let $u\colon X\to \R$ be a continuous function with a $2$-weak upper gradient $\rho_u\in L^2(X)$. 
    \begin{enumerate}[label=\normalfont(\arabic*)]
        \item If $\mathcal A_u$ denotes the union of all non-degenerate components of the level sets $u^{-1}(t)$, $t\in \R$, of $u$, then $\mathcal A_u$ is a Borel set. \label{ca:i_}
        \item For every Borel function $g\colon X\to [0,\infty]$ we have
        $$\int\displaylimits^* \int_{u^{-1}(t)\cap \mathcal A_u}g\, d\mathcal H^1\, dt \leq K \int g\rho_u\, d\mathcal H^2.$$\label{ca:ii_}
    \end{enumerate}
\end{thm}

\begin{proof}
    First we show that $\mathcal A_u$ is a Borel set. We can write
    $$\mathcal A_u= \bigcup_{k=1}^\infty A_k,$$
    where $A_k$ is the union of the components $E$ of $u^{-1}(t)$, $t\in \R$, with $\diam(E)\geq 1/k$. We will show that $A_k$ is closed for each $k\in \N$. Let $\{x_n\}_{n\in \N}$ be a sequence in $A_k$. If $x_n\in E_n\subset u^{-1}(t_n)$, $n\in \N$, then after passing to a subsequence, the continua $E_n$ converge in the Hausdorff sense to a continuum $E$ with $\diam(E)\geq 1/k$. Moreover, after passing to a further subsequence, $t_n$ converges to some $t\in \R$, so $E\subset u^{-1}(t)$. This shows that $E\subset A_k$. Therefore, all limit points of $\{x_n\}_{n\in \N}$ lie in $A_k$, as desired. 

    Let $f\colon \bar \D\to X$ be a weakly $K$-quasiconformal map. By Theorem \ref{theorem:definitions_qc} there exists a $2$-weak upper gradient $\rho_f\in L^2(\bar \D)$ such that
    $$\int_{f^{-1}(E)} \rho_f^2\, d\mathcal H^2 \leq K\mathcal H^2(E)$$
    for each Borel set $E\subset X$. This implies that
    \begin{align}\label{coarea:qc}
        \int (g\circ f) \cdot \rho_f^2\, d\mathcal H^2\leq K \int g \, d\mathcal H^2
    \end{align}
    for each Borel function $g\colon X\to [0,\infty]$. Moreover, for all curves $\gamma$ in $\bar \D$ outside a curve family $\Gamma_0$ with $\mod \Gamma_0=0$ we have (see \cite{HKST:15}*{Prop.\ 6.3.3})
    \begin{align}\label{coarea:upper_gradient}
            \int_{f\circ \gamma}g\, ds\leq \int_{\gamma}(g\circ f)\cdot \rho_f\, ds.
    \end{align}

    Consider the function $v=u\circ f$ on $\bar \D$. Then by \cite{EIR22}*{Lemma 4.5}, $v$ has a $2$-weak upper gradient $\rho_v$ such that for a.e.\ $x\in \bar \D$ we have 
    $$\rho_v(x)\leq (\rho_u\circ f)(x) \cdot \rho_f(x).$$
    In conjunction with \eqref{coarea:qc}, this implies that $\rho_v\in L^2(\bar \D)$, so $v\in W^{1,2}(\D)$, and 
    \begin{align}\label{coarea:gradient}
        |\nabla v(x)|\leq (\rho_u\circ f)(x) \cdot \rho_f(x)
    \end{align}
    for a.e.\ $x\in \D$, because $|\nabla v|$ is the minimal $2$-weak upper gradient of $v$ (see \cite{HKST:15}*{Theorem 7.4.5}). We can extend $v$ by reflection to a continuous function $\widetilde v\in W^{1,2}(U)$ for some neighborhood $U$ of $\bar \D$. By the classical coarea formula (Theorem \ref{theorem:coarea:classical}), the set $v^{-1}(t)=\widetilde{v}^{-1}(t)\cap \bar\D$ has finite Hausdorff $1$-measure for a.e.\ $t\in \R$. Corollary \ref{corollary:parametrizations} implies that for a.e.\ $t\in \R$ each component $E$ of $v^{-1}(t)$ is a Jordan arc or a Jordan curve and can be parametrized by a Lipschitz function $\gamma\colon [a,b]\to E$ that is injective on $[a,b)$.  Moreover, using the classical coarea formula for $\widetilde v$ one can show that for a.e.\ $t\in \R$, each Lipschitz curve $\gamma\colon [a,b]\to v^{-1}(t)$ that is injective on $[a,b)$ lies outside the given curve family $\Gamma_0$ of $2$-modulus zero (cf.\ Lemma \ref{lemma:avoid}); hence $\gamma$ satisfies \eqref{coarea:upper_gradient}. Therefore, the following statements are true for a.e.\ $t\in \R$. 
    \begin{enumerate}
        \item $\mathcal H^1(v^{-1}(t))<\infty$. (Consequence of classical coarea formula.)
        \item Each non-degenerate component $E$ of $v^{-1}(t)$ is a Jordan arc or a Jordan curve and there exists a Lipschitz parametrization $\gamma\colon[a,b]\to E$ that is injective in $[a,b)$.  (Consequence of Corollary \ref{corollary:parametrizations}.)
        \item For each Lipschitz curve $\gamma\colon[a,b]\to v^{-1}(t)$ that is injective on $[a,b)$ and for each Borel function $g\colon X\to [0,\infty]$, we have
        $$\int_{f\circ \gamma} g\, ds \leq \int_{\gamma}(g\circ f)\cdot \rho_f\, ds.$$
        (Consequence of classical coarea formula and \eqref{coarea:upper_gradient}.)
    \end{enumerate}
    We fix a Borel function $g\colon X\to [0,\infty]$, a value $t\in \R$ satisfying the above statements, a non-degenerate component $E$ of $v^{-1}(t)$, and a Lipschitz parametrization $\gamma\colon [a,b]\to E$ that is injective in $[a,b)$. We have
    $$\int_{f(E)}g\, d\mathcal H^1 = \int_{f(|\gamma|)} g\, d\mathcal H^1\leq \int_{f\circ \gamma}g\, ds\leq \int_{\gamma} (g\circ f)\cdot \rho_f\, ds =\int_{E}(g\circ f)\cdot \rho_f\, d\mathcal H^1.$$
    Note that if $G$ is a non-degenerate component of $u^{-1}(t)$, then by the monotonicity of $f$, $f^{-1}(G)$ is a non-degenerate component of $v^{-1}(t)$. Hence, 
    $$\int_{G}g\, d\mathcal H^1 \leq \int_{f^{-1}(G)}(g\circ f)\cdot \rho_f\, d\mathcal H^1.$$
    The finiteness of the Hausdorff $1$-measure of $v^{-1}(t)$ implies that it can have at most countably many non-degenerate components. Summing over all the non-degenerate components gives
    $$\int_{u^{-1}(t)\cap \mathcal A_u}g\, d\mathcal H^1\leq  \int_{v^{-1}(t)}(g\circ f)\cdot \rho_f\, d\mathcal H^1.$$
    We now integrate over $t\in \R$, use the classical coarea formula for $\widetilde v$, and inequalities \eqref{coarea:gradient} and \eqref{coarea:qc}, to obtain
    \begin{align*}
        \int\displaylimits^* \int_{u^{-1}(t)\cap \mathcal A_u} g\, d\mathcal H^1dt &\leq \int \int_{v^{-1}(t)} (g\circ f)\cdot \rho_f\, d\mathcal H^1dt \\&= \int_{\bar  \D} (g\circ f )\cdot\rho_f\cdot |\nabla \widetilde v|\, d\mathcal H^2\\
        &= \int_{\D} (g\circ f )\cdot\rho_f\cdot |\nabla  v|\, d\mathcal H^2\\
        &\leq \int (g\circ f) \cdot (\rho_u\circ f) \cdot \rho_f^2 \, d\mathcal H^2\\&\leq K \int g \rho_u\, d\mathcal H^2. 
    \end{align*}
    This completes the proof. 
\end{proof}

\begin{proof}[Proof of Theorem \ref{theorem:coarea:introduction}]
We write $X$ as the countable union of topological closed disks $X_n$ with $\mathcal H^2(X_n)<\infty$, $n\in \N$. We also consider topological closed disks $Z_n\supset X_n$, so that the topological interior  $\inter_{\top}(Z_n)$ contains  $X_n$. We have $\inter(Z_n)\subset \inter_{\top}(Z_n)\subset Z_n$, where $\inter(Z_n)$ refers to the manifold interior. Therefore the topological closure of $\inter_{\top}(Z_n)$ is precisely the closed disk $Z_n$. Let $u_n=u|_{Z_n}$. We claim that 
\begin{align}\label{coarea:union}
    \mathcal A_u=\bigcup_{n=1}^\infty \mathcal A_{u_n}.
\end{align}
For this, it suffices to show that
\begin{align}\label{coarea:inclusion}
    \mathcal A_u\cap X_n \subset \mathcal A_{u_n}
\end{align}
for each $n\in \N$. Let $x\in \mathcal A_u\cap X_n$ and consider a non-degenerate component $E$ of $u^{-1}(t)$ for some $t\in \R$ such that $x\in E$.  Note that $x$ lies in  $\inter_{\top}(Z_n)$. If $E\subset Z_n$, then $E\subset \mathcal A_{u_n}$ and $x\in \mathcal A_{u_n}$. Suppose that $E$ is not contained in $Z_n$; in this case $E\cap \partial_{\top} Z_n\neq \emptyset$ by the connectedness of $E$. Since $E$ is a generalized continuum (i.e., a locally compact connected set), by \cite{Whyburn}*{(10.1), p.~16}, we conclude that each component of $E\cap {Z_n}$ intersects $\partial_{\top} Z_n$. In particular, the component $E_x$ of $E\cap Z_n$ that contains $x$ must intersect $\partial_{\top} Z_n$, and thus $E_x$ is non-degenerate. We conclude that $E_x\subset \mathcal A_{u_n}$, so $x\in \mathcal A_{u_n}$. The claim is proved. Now, each $\mathcal A_{u_n}$ is a Borel set by Theorem \ref{theorem:coarea}, so $\mathcal A_u$ is Borel measurable by \eqref{coarea:union} and we have established \ref{ca:i}.

Let $g\colon X\to [0,\infty]$ be a Borel function. For $n\in \N$, let $g_n=g\cdot \chi_{X_n\setminus \bigcup_{i=1}^{n-1}X_i}$. Let $x\in \mathcal A_u \cap (X_n\setminus \bigcup_{i=1}^{n-1}X_i)$. Then $x\in \mathcal A_{u_n}$ by \eqref{coarea:inclusion}, so
$$g(x)\chi_{\mathcal A_u}(x)= g_n(x)\chi_{\mathcal A_u}(x) =g_n(x)\chi_{\mathcal A_{u_n}}(x).$$
We conclude that 
\begin{align*}
    g\chi_{\mathcal A_u} = \sum_{n\in \N} g_n\chi_{\mathcal A_{u_n}}. 
\end{align*}
By Theorem \ref{theorem:coarea}, applied to $u_n\colon Z_n\to \R$, and the existence of weakly $(4/\pi)$-quasi\-conformal parametrizations (Theorem \ref{theorem:wqc}), we have
$$\int\displaylimits^* \int_{u^{-1}(t)} g_n \chi_{\mathcal A_{u_n}}\, d\mathcal H^1dt \leq \frac{4}{\pi} \int g_n\rho_u\, d\mathcal H^2 $$
for each $n\in \N$. Thus, upon summing we obtain the claimed inequality \ref{ca:ii}.  

Finally,  part \ref{ca:iii} follows from part \ref{ca:ii} and the coarea inequality for Lipschitz functions. Namely, one applies \ref{ca:ii} to the Borel  function $g\chi_{\mathcal A_u}$ and Theorem \ref{theorem:coarea:classical} to $g\chi_{X\setminus \mathcal A_u}$.
\end{proof}

\bibliography{bibli.bib}

\end{document}